\documentclass[11pt]{amsart}
\usepackage{amsfonts,amssymb}
\usepackage{amsmath}
\usepackage{amsthm,amscd,latexsym,mathrsfs,hyperref}

\newtheorem{theorem}{Theorem}[section]
\newtheorem{corollary}[theorem]{Corollary}
\newtheorem{lemma}[theorem]{Lemma}
\newtheorem{proposition}[theorem]{Proposition}

\newtheorem*{claim}{Claim}

\theoremstyle{definition}
\newtheorem{definition}{Definition}[section]
\newtheorem{remark}{Remark}[section]
\newtheorem{example}{Example}[section]
\newtheorem*{notation}{Notation}

\newcommand{\lra}{\longrightarrow}

\DeclareMathOperator{\argmin}{\mathrm{arg\, min}}

\DeclareMathOperator{\End}{\mathrm{End}}

\DeclareMathOperator{\tr}{tr}
\DeclareMathOperator{\Hess}{Hess}
\newcommand\supp{\mathop{\rm supp}}

\begin{document}

\title[Means \& Karcher equations]{Operator means of probability measures and generalized Karcher equations}
\author[Mikl\'os P\'alfia] {Mikl\'os P\'alfia}
\address{Department of Mathematics, Kyoto University, Kyoto 606-8502, Japan.}
\email{palfia.miklos@aut.bme.hu}

\subjclass[2000]{Primary 47A64, 46L05, Secondary 53C20, 53C35}
\keywords{operator monotone function, operator mean, Karcher mean}


\date{\today}

\begin{abstract}
In this article we consider means of positive bounded linear operators on a Hilbert space. We present a complete theory that provides a framework which extends the theory of the Karcher mean, its approximating matrix power means, and a large part of Kubo-Ando theory to arbitrary many variables, in fact, to the case of probability measures with bounded support on the cone of positive definite operators. This framework characterizes each operator mean extrinsically as unique solutions of generalized Karcher equations which are obtained by exchanging the matrix logarithm function in the Karcher equation to arbitrary operator monotone functions over the positive real half-line. If the underlying Hilbert space is finite dimensional, then these generalized Karcher equations are Riemannian gradients of convex combinations of strictly geodesically convex log-determinant divergence functions, hence these new means are the global minimizers of them, in analogue to the case of the Karcher mean as pointed out. Our framework is based on fundamental contraction results with respect to the Thompson metric, which provides us nonlinear contraction semigroups in the cone of positive definite operators that form a decreasing net approximating these operator means in the strong topology from above.
\end{abstract}

\maketitle

\section{Introduction}
Let $E$ be Hilbert space and $S(E)$ denote the Banach space of bounded linear self-adjoint operators. Let $\mathbb{P}\subseteq S(E)$ denote the cone of positive definite operators on $E$. In this article we are concerned with means of members of $\mathbb{P}$ that enjoy certain attractive properties that recently became important from the point of view of averaging in the finite dimensional case, see for example \cite{AF,fillard,karcher,Ba08,BBC}. Usually the main difficulties here arise from the required property of operator monotonicity i.e., our means must be monotone with respect to the positive definite order on $\mathbb{P}$. The 2-variable theory of such functions is relatively well understood: each such function is represented by an operator monotone function according to the theory of Kubo-Ando \cite{kubo}. However in the several variable case we have no such characterization of operator monotone functions.

When $E$ is finite dimensional, then there are additional geometrical structures on $\mathbb{P}$ that are used to define n-variable operator means \cite{ando,bhatiaholbrook,lawsonlim,moakher}. In this setting $\mathbb{P}$ is just the cone of positive definite $n$-by-$n$ Hermitian matrices, where $n$ is the dimension of $E$. It is a smooth manifold as an open subset of the vector space of $n$-by-$n$ Hermitian matrices (which is just $S(E)$ in this case) and has a Riemannian symmetric space structure $\mathbb{P}\cong\mathrm{GL}(E)/\mathbb{U}$, where $\mathbb{U}$ is the unitary group and $\mathrm{GL}(E)$ is the general linear group over $E$ \cite{bhatia,bhatia2}. This symmetric space is nonpositively curved, hence a unique minimizing geodesic between any two points exists \cite{bhatia2}. The midpoint operation on this space, which is defined as taking the middle point of the geodesic connecting two points, is the geometric mean of two positive definite matrices \cite{bhatia2}. The Riemannian distance on the manifold $\mathbb{P}$ is of the form
\begin{equation*}
d(A,B)=\sqrt{Tr\log^2(A^{-1}B)}.
\end{equation*}
Then the (weighted) multivariable geometric mean or Karcher mean of the k-tuple ${\Bbb A}:=(A_1,\ldots,A_k)\in\mathbb{P}^k$ with respect to the positive probability vector $\omega:=(w_1,\ldots,w_k)$ is defined as the center of mass
\begin{equation}\label{riemmean}
\Lambda(w_1,\dots,w_k;A_{1},\dots,A_{k})=\underset{X\in \textit{P}(n,\mathbb{C})}{\argmin}\sum_{i=1}^k w_id^2(X,A_i).
\end{equation}
The Karcher mean was first considered in this setting in \cite{moakher,bhatiaholbrook}. The Karcher mean $\Lambda(\omega;{\Bbb A})$ is also the unique positive definite solution of the corresponding critical point equation called the Karcher equation
\begin{equation}\label{eq:intr1}
\sum_{i=1}^k w_i\log(X^{-1}A_i)=0,
\end{equation}
where the gradient of the function in the minimization problem \eqref{riemmean} is appears on the left hand side \cite{moakher}. The (entry-wise) operator monotonicity of this mean in ${\Bbb A}$ with respect to a fixed $\omega$, that is $\mathbb{A}\leq\mathbb{B}$ if and only if $A_i\leq B_i$ for $1\leq i\leq k$ implying $\Lambda(\omega;{\Bbb A})\leq \Lambda(\omega;{\Bbb B})$ in the positive definite order, was unknown until it has been first proved in \cite{lawsonlim}. Later in \cite{limpalfia} a one parameter family $P_s(\omega,\mathbb{A})$ of operator monotone means, called the matrix power means, has been constructed as the unique solution of the nonlinear operator equation
\begin{equation}\label{intr:powermeanequ2}
X=\sum_{i=1}^{k}w_{i}X\#_sA_i
\end{equation}
for fixed $s\in[-1,1], w_i>0, \sum_{i=1}^kw_i=1$ and $A_i\in\mathbb{P}$, where
\begin{equation*}
A\#_sB=A^{1/2}\left(A^{-1/2}BA^{-1/2}\right)^{s}A^{1/2}
\end{equation*}
is the weighted two-variable geometric mean. Then in \cite{limpalfia} it has been established that
\begin{equation*}
\lim_{s\to 0+}P_s(\omega,\mathbb{A})=\Lambda(\omega;{\Bbb A})
\end{equation*}
proving once more monotonicity and many other additional properties of $\Lambda$. Later the same approach based on the matrix power means has been adapted in \cite{lawsonlim1}, to extend $\Lambda(\omega;{\Bbb A})$ to the infinite dimensional setting as the unique solution of \eqref{eq:intr1}, even though there exists no Riemannian metric in this case, hence nor geodesically convex potential functions. These additional properties make the Karcher mean $\Lambda(\omega;{\Bbb A})$ the geometric mean to choose among the numerous different geometric means \cite{bhatiaholbrook,lawsonlim12}. Form the geometric and analytic point of view the two characterizations provided by \eqref{riemmean} and \eqref{eq:intr1} may be the most significant.

In this paper we extend the theory of the Karcher mean and the approximating matrix power means appearing in\cite{lawsonlim1,lawsonlim12,limpalfia} in three different ways. Firstly we generalize the setting of \eqref{eq:intr1} by taking any operator monotone function $f:(0,\infty)\mapsto\mathbb{R}$ instead of just the special one $f(x)=\log(x)$. Here operator monotonicity of $f$ is understood as $A\leq B$ implying $f(A)\leq f(B)$ for any $A,B\in\mathbb{P}$. The second generalization is that instead of sums for n-tuples of operators ${\Bbb A}$ in \eqref{eq:intr1} we consider integrals with respect to probability measures supported on $\mathbb{P}$. So our generalization of \eqref{eq:intr1} has the form
\begin{equation}\label{eq:intr2}
\int_{\mathbb{P}} f(X^{-1}A)d\mu(A)=0
\end{equation}
for a probability measure $\mu$ on $\mathbb{P}$. Let $\mathfrak{L}$ denote the set of all operator monotone functions $f:(0,\infty)\mapsto\mathbb{R}$ such that $f(1)=0, f'(1)=1$. Let $\mathscr{P}(\mathfrak{L}\times\mathbb{P})$ denote the set of Borel probability measures with bounded support on the product space $\mathfrak{L}\times\mathbb{P}$. Then our first result is the following.
\begin{theorem}\label{maintheorem1}
Let $\mu\in\mathscr{P}(\mathfrak{L}\times\mathbb{P})$. Then the equation
\begin{equation}\label{eq:intr3}
\int_{\mathfrak{L}\times\mathbb{P}}f\left(X^{-1/2}AX^{-1/2}\right)d\mu(f,A)=0,
\end{equation}
which we call the generalized Karcher equation, has a unique solution in $\mathbb{P}$ that is also operator monotone in the extended sense of Definition~\ref{D:measureOrder} below, which provides a positive definite partial order for measures.
\end{theorem}
We prove this by suitably generalizing the argumentation in \cite{lawsonlim1}. We establish a one parameter family of generalized matrix power means as fixed points of strict contractions, with suitable properties like operator monotonicity, such that they form a decreasing net in the partial positive definite order approximating the solution of \eqref{eq:intr3} from above in the strong topology. Moreover these generalized matrix power means, called induced operator means, are unique solutions of \eqref{eq:intr3} as well, hence fitting into the greater picture. By establishing many attractive properties, like operator monotonicity, of these fixed points, we prove that these properties are preserved in the limit i.e., a solution of \eqref{eq:intr3} has these properties as well. Then we prove that the left hand side of \eqref{eq:intr3} occurs as a Riemannian gradient of certain divergence functions on $\mathbb{P}$ when $E$ is finite dimensional, using the integral formula
\begin{equation}\label{eq:intr4}
f(x)=\int_{[0,1]}\l_s(x)d\nu(s)
\end{equation}
valid for any $f\in\mathfrak{L}$, where $\l_s(x):=\frac{x-1}{(1-s)x+s}$ and $\nu$ is a uniquely determined probability measure on $[0,1]$. In particular we obtain a generalized form of \eqref{riemmean} by means of strictly geodesically convex divergence functions
\begin{equation}\label{eq:div}
LD^s(X,A)=\frac{1}{s(1-s)}\tr\left\{\log[(1-s)A+sX]-\log X\#_{1-s}A\right\}
\end{equation}
where $s\in[0,1]$, $X\#_{1-s}A$ is the weighted geometric mean of $X,A\in\mathbb{P}$. These divergence functions \eqref{eq:div} also appeared recently in \cite{chebbi}. Let $\mathscr{P}([0,1]\times\mathbb{P})$ denote the set of Borel probability measures on $[0,1]\times\mathbb{P}$. Then our second main result is the following.
\begin{theorem}\label{maintheorem2}
Let $\mathbb{P}$ be finite dimensional and $\mu\in\mathscr{P}([0,1]\times\mathbb{P})$. Then the minimization problem
\begin{equation}\label{eq:intr5}
\argmin_{X\in\mathbb{P}}\int_{[0,1]\times\mathbb{P}}LD^s(X,A)d\mu(s,A)
\end{equation}
has a unique positive definite solution in $\mathbb{P}$. Moreover it is the unique solution of the corresponding generalized Karcher equation
\begin{equation*}
\int_{[0,1]\times\mathbb{P}}X^{1/2}\l_s(X^{-1/2}AX^{-1/2})X^{1/2}d\mu(s,A)=0,
\end{equation*}
which is the critical point equation of \eqref{eq:intr5} with respect to the Riemannian metric of the symmetric space $\mathbb{P}$.
\end{theorem}
According to formula \eqref{eq:intr4}, the above theorem exhausts all possible generalized Karcher equations appearing in Theorem~\ref{maintheorem1}. Also the combination of Theorem~\ref{maintheorem1} and Theorem~\ref{maintheorem2}, as a byproduct, yields a new property of operator monotone functions in $\mathfrak{L}$: they are characterized as Riemannian gradients of divergence functions.

Theorem~\ref{maintheorem1} and Theorem~\ref{maintheorem2} are extensions of the results known for the geometric (or Karcher) mean $\Lambda(\omega;\mathbb{A})$ \cite{lawsonlim1,lawsonlim12,limpalfia,moakher}, in which case $f(x)=\log(x)$ in Theorem~\ref{maintheorem1} and in Theorem~\ref{maintheorem2} the measure $d\mu(s,A)=d(\nu\times\sigma)(s,A)$, where $d\nu(s)=ds$ and $d\sigma(A)$ is finitely supported over $A_i\in\mathbb{P}$ with corresponding weight $w_i$ for $1\leq i\leq k$. Also the matrix power means $P_s(\omega,\sigma)$ fit in the picture, in which case $f(x)=\frac{x^s-1}{s}$ in Theorem~\ref{maintheorem1}. Also our theorems extend the log-determinant means of positive matrices in \cite{chebbi}. Hence the unique solutions appearing in Theorem~\ref{maintheorem1} and Theorem~\ref{maintheorem2} not only provide us with new operator means, but unifies these existing approaches which - to our knowledge - is the first generally applicable extension of two-variable Kubo-Ando operator means \cite{kubo} to several (noncommutative) variables and measures. Many nice properties of these generalized operator means are proved here including operator monotonicity, see Theorem~\ref{lambdaprop} below.

The proof of Theorem~\ref{maintheorem2} is based on a contraction principle in the cone $\mathbb{P}$ with a complete metric space structure given by Thompson's part metric
\begin{equation*}
d_\infty(A,B)=\max\left\{\log M(A/B),M(B/A)\right\}
\end{equation*}
for any $A,B\in\mathbb{P}$, where $M(A/B)=\inf\{\alpha:A\leq\alpha B\}$. The argument is an elaborate generalization of the one given in \cite{lawsonlim1} which is motivated by the one in \cite{limpalfia} based on the approximation of the Karcher mean with the matrix power means which are themselves fixed points of one parameter families of strict contractions. Also en route this process, new auxiliary results are proved for the Thompson metric as well which extend various earlier ones in \cite{lawsonlim0}.

The paper is organized as follows. In section 2 we introduce the divergence functions \eqref{eq:div}, calculate their gradients and prove their geodesic convexity along with new integral formulas for members in $\mathfrak{L}$. These results will provide Theorem~\ref{T:optimize} which gives the first half of Theorem~\ref{maintheorem1}. Then in section 3 we prove new integral formulas for operator means in the sense of Kubo-Ando. In section 4 we prove fundamental contraction results for operator means with respect to the Thompson metric and in section 5 we use these to build up an analogous theory of generalized matrix power means called induced operator means in our setting. Then in section 6 we prove Theorem~\ref{maintheorem2} as a combined effort of Theorem~\ref{inducedconv}, Theorem~\ref{karchersatisfied} and Theorem~\ref{uniquegeneralizedkarcher} which also concludes the second half of Theorem~\ref{maintheorem1} and Theorem~\ref{maintheorem2}. In section 7 we briefly discuss the results obtained.

\section{Convex cone of log-determinant divergences over $\mathbb{P}$}
Suppose now that the Hilbert space $E$ is finite dimensional so that $\mathbb{P}$ is isomorphic to the cone of positive definite matrices. In this case one can introduce the so called logarithmic-barrier function
\begin{equation*}
f(X)=-\log\det X=-\tr\log X
\end{equation*}
which is strictly convex on $\mathbb{P}$ \cite{bhatia}. Using this function we introduce the following one parameter family of divergence functions on $\mathbb{P}$ also studied in \cite{chebbi}:
\begin{equation}
LD^s(X,A)=\frac{1}{s(1-s)}\tr\left\{\log[(1-s)A+sX]-\log X\#_{1-s}A\right\}
\end{equation}
for $s\in[0,1]$ and $X,A\in\mathbb{P}$ where the cases $s=0,1$ are defined by a continuity argument and
\begin{equation}\label{eq:geom:1}
X\#_{1-s}A=X^{1/2}\left(X^{-1/2}AX^{-1/2}\right)^{1-s}X^{1/2}=X\left(X^{-1}A\right)^{1-s}
\end{equation}
is the weighted geometric mean. Indeed, one can check that
\begin{eqnarray*}
LD^1(X,A)&=&\tr\left\{X^{-1}A-I-\log(X^{-1}A)\right\}\\
LD^0(X,A)&=&\tr\left\{A^{-1}X-I-\log(A^{-1}X)\right\}
\end{eqnarray*}
by computing the limits $\lim_{s\to 1-}LD^s(X,A)$ and $\lim_{s\to 0+}LD^s(X,A)$. Moreover simple calculation show that
\begin{eqnarray*}
LD^s(X,A)&=&\frac{1}{s(1-s)}\log\frac{\det[(1-s)A+sX]}{\det(X\#_{1-s}A)}\\
&=&\frac{1}{s(1-s)}\tr\left\{\log[(1-s)X^{-1}A+s]-\log (X^{-1}A)^{1-s}\right\}.
\end{eqnarray*}
The formulas defining $LD^s(X,A)$ show that indeed it is a \emph{divergence function}, that is, it satisfies
\begin{itemize}
\item[(I)] $LD^s(X,A)\geq 0$,
\item[(II)] $LD^s(X,A)=0$ if and only if $X=A$,
\end{itemize}
where (I) essentially follows from the operator inequality between the arithmetic and geometric means
$$sX+(1-s)A\geq X\#_{1-s}A$$
and the operator monotonicity of $\log$, see for example \cite{bhatia2}. Property (I) and (II) tells us that $LD^s(X,A)$ is a divergence function on $\mathbb{P}$. Now clearly any convex combinations of the form
$$\int_{[0,1]}LD^s(X,A)d\nu(s)$$
also satisfies (I) and (II) for any probability measure $\nu$ over the closed interval $[0,1]$. Here we can consider arbitrary positive measures over $[0,1]$, but we will see that the normalization condition $\int_{[0,1]}d\nu(s)=1$ does not restrict our further analysis, since we will consider the point where these functions attain their minimum on $\mathbb{P}$.

We are interested in the convexity properties of $\int_{[0,1]}LD^s(X,A)d\nu(s)$. More precisely the geodesic convexity with respect to the Riemannian metric
\begin{equation}\label{Rmetric}
\left\langle X,Y\right\rangle_A=\tr\left\{A^{-1}XA^{-1}Y\right\}
\end{equation}
where $A\in\mathbb{P}$ and $X,Y\in S(E)$, where $S(E)$ in this case is just the vector space of Hermitian matrices. It is well known that $\mathbb{P}$ with the Riemannian metric is a Hadamard manifold i.e., it has nonpositive sectional curvature \cite{bhatia2}. Therefore its geodesics connecting any two points $A,B\in\mathbb{P}$ are unique and are given by the geometric mean $A\#_tB$. The Levi-Civita connection with respect to the Riemannian metric is of the form
\begin{equation}\label{LeviCivita}
\nabla_{X_A}Y_A=DY[A][X_A]-\frac{1}{2}\left(X_AA^{-1}Y_A+Y_AA^{-1}X_A\right)
\end{equation}
see for example \cite{Lar}.

Let us consider the set of operator monotone functions $f:(0,\infty)\mapsto \mathbb{R}$. Operator monotonicity means that for any two Hermitian matrices $X\leq Y$ we have $f(X)\leq f(Y)$. Such functions have strong analytic properties, each such function has a unique integral representation \cite{bhatia} of the form
\begin{equation}\label{integralrepr}
f(x)=\alpha+\beta x+\int_{0}^{\infty}\frac{\lambda}{\lambda^2+1}-\frac{1}{\lambda+x}d\mu(\lambda)\text{,}
\end{equation}
where $\alpha$ is a real number, $\beta\geq 0$ and $\mu$ is a unique positive measure on $[0,\infty)$ such that
\begin{equation}
\int_{0}^{\infty}\frac{1}{\lambda^2+1}d\mu(\lambda)<\infty\text{.}
\end{equation}
Notice that the set of all such functions is a convex cone. A subset of this cone will be of our interest:
$$\mathfrak{L}:=\left\{f:(0,\infty)\mapsto\mathbb{R}, f\text{ is operator monotone }, f(1)=0, f'(1)=1\right\}.$$
The following is a simple corollary of Bendat and Sherman's result \cite{bendatsherman}, see Theorem 3.7 in \cite{hansen3}.

\begin{corollary}\label{C:intLrepr}
The function $f\in\mathfrak{L}$ if and only if there exists a unique probability measure $\nu$ on $[0,1]$ such that
\begin{equation}\label{intreprL2}
f(x)=\int_{[0,1]}\frac{x-1}{(1-s)x+s}d\nu(s).
\end{equation}
\end{corollary}

We can also consider
\begin{eqnarray*}
\int_{[0,1]}\frac{x-1}{(1-s)x+s}d\nu(s)&=&\int_{[0,\infty]}\frac{x-1}{\frac{x}{\lambda+1}+\frac{\lambda}{\lambda+1}}d\nu\left(\frac{\lambda}{\lambda+1}\right)\\
&=&\int_{[0,\infty]}\lambda+1-\frac{(\lambda+1)^2}{\lambda+x}d\nu\left(\frac{\lambda}{\lambda+1}\right)
\end{eqnarray*}
with $d\eta(\lambda):=d\nu\left(\frac{\lambda}{\lambda+1}\right)$, so that we obtain the following similar corollary as well:

\begin{corollary}\label{C:lambdaintchar}
The function $f\in\mathfrak{L}$ if and only if
\begin{equation}
f(x)=\int_{[0,\infty]}\lambda+1-\frac{(\lambda+1)^2}{\lambda+x}d\eta(\lambda)
\end{equation}
where $\eta$ is a unique probability measure on the closed, compact interval $[0,\infty]$.
\end{corollary}

The above integral characterizations will be useful tools also in later sections. One can easily see that
\begin{equation}\label{eq:logbounds}
1-x^{-1}\leq \frac{x-1}{(1-s)x+s}\leq x-1
\end{equation}
holds for all $s\in[0,1]$ and $x>0$, by comparing the functions on the disjoint intervals $(0,1)$ and $[1,\infty)$. Hence after integration we conclude that
\begin{equation}\label{eq:logbounds2}
1-x^{-1}\leq f(x)\leq x-1
\end{equation}
holds for all $f\in\mathfrak{L}$. Now we give yet another characterization of the set $\mathfrak{L}$.
\begin{definition}
Let $\mathfrak{D}_A$ denote the simplex of divergence functions on $\mathbb{P}$ of the form
$$F(X)=\int_{[0,1]}LD^s(X,A)d\nu(s)$$
for any probability measure $\nu$ over the closed interval $[0,1]$ and $A\in\mathbb{P}$.
\end{definition}

Note that the above integral exists, since for fixed $X,A\in\mathbb{P}$ the real function $s\mapsto LD^s(X,A)$ is continuous over the compact interval $[0,1]$, hence bounded and strongly measurable so it is integrable by dominated convergence theorem or by Theorem 11.8 in \cite{aliprantis}. The same reasoning ensures the existence of other integrals over $[0,1]$ that are considered in this section.

\begin{theorem}\label{T:RiemaGrad}
Let $F\in\mathfrak{D}_A$ be represented by a probability measure $\nu$ over the closed interval $[0,1]$. Then the Riemannian gradient is
\begin{equation}\label{eq:RiemaGrad}
\begin{split}
\nabla F(X)&=-\int_{[0,1]}X(X^{-1}A-I)\left[(1-s)X^{-1}A+sI\right]^{-1}d\nu(s)\\
&=-Xf(X^{-1}A)=-X^{1/2}f(X^{-1/2}AX^{-1/2})X^{1/2}
\end{split}
\end{equation}
with respect to the Riemannian metric \eqref{Rmetric} and $f\in\mathfrak{L}$ represented by $\nu$ as in \eqref{intreprL2}.
\end{theorem}
\begin{proof}
The Riemannian gradient $\nabla F(X)$ of $F(X)$ is defined by the relation
\begin{equation*}
\left.\frac{\partial}{\partial t}F(X+Vt)\right|_{t=0}=\left\langle \nabla F(X),V\right\rangle_X
\end{equation*}
for any $V$ in the tangent space at $X$. I.e., we compute
\begin{equation*}
\begin{split}
\frac{\partial}{\partial t}LD^s&\left.(X+Vt,A)\right|_{t=0}=\\
=&\frac{1}{s(1-s)}\tr\left\{D\log\left((1-s)X^{-1}A+sI\right)\left(-(1-s)X^{-1}VX^{-1}A\right)\right.\\
&\left.-(1-s)D\log\left(X^{-1}A\right)\left(-X^{-1}VX^{-1}A\right)\right\}
\end{split}
\end{equation*}
where $D\log:\mathbb{P}\times S(E)\to S(E)$ is the Fr\'echet derivative of $\log$ which is linear in the second variable. Moreover since $\log$ is an analytic function on $(0,\infty)$ we have by the linearity, cyclic property of the trace and the Riesz-Dunford functional calculus \cite{weidman}, that for any $X,Y$ we have that
$$\tr\left\{D\log(X)(Y)\right\}=\tr\left\{D\log(X)(I)Y\right\}.$$
Further calculation with this gives that
\begin{equation*}
\begin{split}
\frac{\partial}{\partial t}LD^s&\left.(X+Vt,A)\right|_{t=0}=\\
=&\frac{1}{s}\tr\left\{D\log\left(X^{-1}A\right)\left(I\right)X^{-1}VX^{-1}A\right.\\
&\left.-D\log\left((1-s)X^{-1}A+sI\right)\left(I\right)X^{-1}VX^{-1}A\right\}\\
=&\frac{1}{s}\tr\left\{A^{-1}XX^{-1}VX^{-1}A-\left[(1-s)X^{-1}A+sI\right]^{-1}X^{-1}VX^{-1}A\right\}\\
=&\frac{1}{s}\tr\left\{VX^{-1}-X^{-1}A\left[(1-s)X^{-1}A+sI\right]^{-1}X^{-1}V\right\},
\end{split}
\end{equation*}
hence
$$\nabla F(X)=\frac{1}{s}X\left\{I-\left[(1-s)I+s(X^{-1}A)^{-1}\right]^{-1}\right\}.$$
Now simple calculation shows that
$$\frac{1-\left[1-s+sx^{-1}\right]^{-1}}{s}=-\frac{x-1}{(1-s)x+s},$$
and that the whole calculation also holds even if $s=0$ or $1$. Since Bochner integration is exchangeable with the linear differential operator of Fr\'echet differentiation, we get from the above that \eqref{eq:RiemaGrad} holds.
\end{proof}

\begin{remark}
The proof of Theorem~\ref{T:RiemaGrad} also shows that actually every $f\in\mathfrak{L}$ corresponds to a Riemannian gradient of a divergence function in $\mathfrak{D}_A$.
\end{remark}

We say that a map $f:\mathbb{P}\to\mathbb{R}$ is geodesically convex with respect to the metric \eqref{Rmetric} if it is convex along any geodesic $\gamma:[0,1]\to\mathbb{P}$ i.e., the function $f(\gamma(t))$ is convex:
$$f(\gamma(t))\leq (1-t)f(\gamma(0))+tf(\gamma(1))$$
for $t\in[0,1]$. Similarly $f$ is strictly geodesically convex if additionally
$$f(\gamma(t))<(1-t)f(\gamma(0))+tf(\gamma(1))$$
for $t\in[0,1]$. It is well known that $f$ is geodesically convex if and only if the Riemannian Hessian $\Hess f(X):S(E)\times S(E)\to \mathbb{R}$ is positive semidefinite, similarly $f$ is strictly geodesically convex if and only if $\Hess f(X)(\cdot,\cdot)$ is positive definite \cite{karcher2,papadopoulos}. Moreover if one can show that $\Hess f(X)>m>0$ on some bounded geodesically convex set, then it follows that $f$ is uniformly convex i.e.,
$$f(\gamma(t))\leq (1-t)f(\gamma(0))+tf(\gamma(1))-\frac{m}{2}t(1-t)d(\gamma(0),\gamma(1))^2$$
for all geodesics $\gamma$ lying entirely in the bounded gedesically convex set. In the above inequality $d$ is the Riemannian distance function.

\begin{theorem}\label{T:RiemHess}
Every $F\in\mathfrak{D}_A$ is strictly geodesically convex function with respect to the metric \eqref{Rmetric}.
\end{theorem}
\begin{proof}
The idea is to show that the Riemannian Hessian of $F$ is positive definite i.e., $\Hess F(X)(V,V)>0$ for any nonzero $V\in S(E)$ in the tangent space at $X$. The Riemannian Hessian of $F$ is defined by
$$\Hess F(X)(V,W)=\left\langle \nabla_V(\nabla F(X)),W\right\rangle_X$$
as a bilinear form acting on the tangent space at $X$, where $\nabla_V$ is the covariant derivative given by the Levi-Civita connection \eqref{LeviCivita}, see for example \cite{spivak}. By Theorem~\ref{T:RiemaGrad} we have that $\nabla F(X)=-Xf(X^{-1}A)$ with an $f\in\mathfrak{L}$, so we have
\begin{eqnarray*}
\nabla_V(\nabla F(X))&=&-Vf(X^{-1}A)+XDf(X^{-1}A)(X^{-1}VX^{-1}A)\\
&&+\frac{1}{2}\left[Xf(X^{-1}A)X^{-1}V+Vf(X^{-1}A)\right].
\end{eqnarray*}
Hence we have
\begin{eqnarray*}
\Hess F(X)(V,V)&=&\left\langle \nabla_V(\nabla F(X)),V\right\rangle_X\\
&=&\tr\left\{X^{-1}\left\{-Vf(X^{-1}A)+XDf(X^{-1}A)(X^{-1}VX^{-1}A)\right.\right.\\
&&\left.\left.+\frac{1}{2}\left[Xf(X^{-1}A)X^{-1}V+Vf(X^{-1}A)\right]\right\}X^{-1}V\right\}.
\end{eqnarray*}
By the cyclic property of the trace the above is equivalent to
\begin{eqnarray*}
\Hess F(X)(V,V)&=&\tr\left\{Df(X^{-1}A)(X^{-1}VX^{-1}A)X^{-1}V\right\}.
\end{eqnarray*}
Now we make use of the integral representation for $f$ given in Corollary~\ref{C:lambdaintchar} in the form
$$f(x)=\int_{[0,\infty]}\lambda+1-\frac{(\lambda+1)^2}{\lambda+x}d\eta(\lambda).$$
By exchanging the Fr\'echet derivative with the integral we obtain that
\begin{equation*}
\begin{split}
Df(X^{-1}A)&(X^{-1}VX^{-1}A)=\nu(\{\infty\})X^{-1}VX^{-1}A\\
&+\int_{[0,\infty)}(\lambda I+X^{-1}A)^{-1}X^{-1}VX^{-1}A(\lambda I+X^{-1}A)^{-1}d\eta(\lambda),
\end{split}
\end{equation*}
where we separated the term corresponding to $\{\infty\}$. Hence by the linearity of the trace we have
\begin{equation*}
\begin{split}
&\Hess F(X)(V,V)=\nu(\{\infty\})\tr\left\{X^{-1}VX^{-1}AX^{-1}V\right\}\\
&+\int_{[0,\infty)}\tr\left\{(\lambda I+X^{-1}A)^{-1}X^{-1}VX^{-1}A(\lambda I+X^{-1}A)^{-1}X^{-1}V\right\}d\eta(\lambda).
\end{split}
\end{equation*}
Hence it suffices to prove now that the expressions
\begin{eqnarray*}
c(V)&:=&\tr\left\{X^{-1}VX^{-1}AX^{-1}V\right\}>0,\\
p(V)&:=&\tr\left\{(\lambda I+X^{-1}A)^{-1}X^{-1}VX^{-1}A(\lambda I+X^{-1}A)^{-1}X^{-1}V\right\}>0
\end{eqnarray*}
for all nonzero $V\in S(E)$, $X,A\in\mathbb{P}$ and $\lambda \geq 0$. Using the notation $P=X^{-1/2}AX^{-1/2}$, $H=X^{-1/2}VX^{-1/2}$, the cyclic property of the trace and that $(\lambda I+P)^{-1/2}$ commutes with $P^{1/2}$, we get that
\begin{eqnarray*}
p(V)&=&\tr\left\{(\lambda I+P)^{-1}HP(\lambda I+P)^{-1}H\right\}\\
&=&\tr\left\{(\lambda I+P)^{-1/2}H(\lambda I+P)^{-1/2}P^{1/2}\right.\\
&&\left.P^{1/2}(\lambda I+P)^{-1/2}H(\lambda I+P)^{-1/2}\right\}\\
&=&\left\|(\lambda I+P)^{-1/2}H(\lambda I+P)^{-1/2}P^{1/2}\right\|_2^2,\\
c(V)&=&\tr\left\{HPH\right\}\\
&=&\tr\left\{HP^{1/2}P^{1/2}H\right\}\\
&=&\left\|HP^{1/2}\right\|_2^2
\end{eqnarray*}
where $\|T\|_2=\sqrt{tr\left\{TT^*\right\}}$ is the Frobenius or Hilbert-Schmidt norm of $T\in\End(E)$. Since $(\lambda I+P)^{-1/2}$ and $P^{1/2}$ are nonsingular, moreover $H$ is nonzero, it follows that
\begin{eqnarray*}
\left\|(\lambda I+P)^{-1/2}H(\lambda I+P)^{-1/2}P^{1/2}\right\|_2^2&>&0,\\
\left\|HP^{1/2}\right\|_2^2&>&0
\end{eqnarray*}
hence $p(V)>0$ and $c(V)>0$ which yields $\Hess F(X)(V,V)>0$ since $\eta$ is a positive measure.
\end{proof}

Theorem~\ref{T:RiemHess} has many important consequences as we will see shortly. Before that let us investigate under what circumstances is the Hessian bounded away from zero in the previous Theorem~\ref{T:RiemHess}. By simple arguments one can see that
\begin{eqnarray*}
\left\|(\lambda I+P)^{-1/2}H(\lambda I+P)^{-1/2}P^{1/2}\right\|_2^2&>&\min\{1/(\lambda+p_i)\}\min\{p_i^{1/2}\}\left\|H\right\|_2^2\\
\left\|HP^{1/2}\right\|_2^2&>&\min\{p_i^{1/2}\}\left\|H\right\|_2^2
\end{eqnarray*}
where $p_i$ are the eigenvalues of $P$. If $\lambda$ is finite, and eigenvalues of $A,X$ are bounded i.e., $X,A$ are in a bounded metric ball and the measure $\eta$ is supported over some closed interval $[a,b]\subset [0,\infty]$, then the above bounds on the right are strictly greater then $0$ for unit length directions $V$, moreover can be chosen uniformly over the bounded set and the interval $[a,b]$. This shows the following.

\begin{corollary}\label{C:strictlyconvex}
The Hessian $\Hess F(X)(\cdot,\cdot)$ in Theorem~\ref{T:RiemHess} has eigenvalues greater than some $c>0$ on geodescially bounded sets i.e., every $F\in\mathfrak{D}_A$ is a uniformly geodesically convex function with respect to the metric \eqref{Rmetric} on bounded metric balls.
\end{corollary}

Now Theorem~\ref{T:RiemHess} and Corollary~\ref{C:strictlyconvex} implies the following result. Let $\mathscr{P}(\mathbb{P})$ denote the set of all Radon probability measures over $\mathbb{P}$ with bounded support in $\mathbb{P}$. The integrals in the following theorem exist by the continuity and boundedness of the integrands over the support of the measures, since continuity ensures strong measurability, while boundedness (which follows from continuity over the compact support of the measures) ensures integrability using the dominated convergence theorem or Theorem 11.8 in \cite{aliprantis}.

\begin{theorem}\label{T:optimize}
Let $\sigma\in\mathscr{P}(\mathbb{P})$ and let $\nu$ be a probability measure on $[0,1]$. Let $C\subseteq\mathbb{P}$ be a closed, bounded geodesically convex set. Then 
\begin{itemize}
\item[I.] The solution of the optimization problem
\begin{equation}
\min_{X\in C}\int_{\mathbb{P}}\int_{[0,1]}LD^s(X,A)d\nu(s)d\sigma(A)
\end{equation}
exists and is unique in $C$.
\item[II.] If the global optimization problem
\begin{equation}\label{eq:optimize1}
\min_{X\in \mathbb{P}}\int_{\mathbb{P}}\int_{[0,1]}LD^s(X,A)d\nu(s)d\sigma(A)
\end{equation}
has a solution, then it is unique and satisfies the nonlinear operator equation
\begin{equation}\label{eq:optimize2}
\int_{\mathbb{P}}X^{1/2}f(X^{-1/2}AX^{-1/2})X^{1/2}d\sigma(A)=0
\end{equation}
where $f(x)=\int_{[0,1]}\frac{x-1}{(1-s)x+s}d\nu(s)$ and $f\in\mathfrak{L}$.
\end{itemize}
\end{theorem}
\begin{proof}
The uniqueness of the solution of the global problem follows from the strict convexity of $LD^s(X,A)$ in Theorem~\ref{T:RiemHess}. Moreover if the minimizer exists then it is a critical point of the gradient of the objective function which is explicitly calculated in Theorem~\ref{T:RiemaGrad}.

The existence of the unique solution of the local problem follows from the continuity and uniform convexity of $LD^s(X,A)$ in bounded geodesically convex sets according to Corollary~\ref{C:strictlyconvex} and a standard optimization argument for strongly convex functions bounded from below, see for example Theorem 1.7 in \cite{sturm}.
\end{proof}

Theorem~\ref{T:optimize} introduces the important notions that will be investigated in the remaining sections of the paper. The gradient equation \eqref{eq:optimize2} is of fundamental importance in this paper.
\begin{definition}[Generalized Karcher equation]\label{D:genKarcherequ}
Let $f\in\mathfrak{L}$ and $\sigma\in\mathscr{P}(\mathbb{P})$. Then the nonlinear operator equation
\begin{equation*}
\int_{\mathbb{P}}X^{1/2}f(X^{-1/2}AX^{-1/2})X^{1/2}d\sigma(A)=0
\end{equation*}
for $X\in\mathbb{P}$ is called the generalized Karcher equation for the function $f$.
\end{definition}

Theorem~\ref{T:optimize} is just not strong enough for our purposes to show the existence of the global minimizer
\begin{equation*}
\argmin_{X\in \mathbb{P}}\int_{\mathbb{P}}\int_{[0,1]}LD^s(X,A)d\nu(s)d\sigma(A),
\end{equation*}
nor to show additional important properties satisfied by the unique solutions. The rest of this paper will be devoted to this existence problem, along with establishing the properties enjoyed by the solution, which are finally concluded in Theorem~\ref{finalconclusion}. We will study this problem indirectly by instead looking at the generalized Karcher equation
\begin{equation*}
\int_{\mathbb{P}}X^{1/2}f(X^{-1/2}AX^{-1/2})X^{1/2}d\sigma(A)=0
\end{equation*}
corresponding to the minimization problem and show the existence of the unique solution of it in the more general setting when $\mathbb{P}$ is possibly over an infinite dimensional Hilbert space $E$. This approach requires us to build a widely applicable machinery which relies on principles of nonlinear contraction semigroups in the ordered cone $\mathbb{P}$.

From now on we abandon the finite dimensional case and let $E$ be infinite dimensional so that $\mathbb{P}$ is the full cone of positive definite operators.

\section{Operator means and representations}
Let us recall the definition of operator (or matrix) mean from \cite{kubo}:
\begin{definition}\label{symmean}
A two-variable function \textit{M}:
$ \mathbb{P}\times \mathbb{P}\mapsto \mathbb{P}$ is called
a matrix or operator mean if
\begin{enumerate}
\renewcommand{\labelenumi}{(\roman{enumi})}
\item $M(I,I)=I$ where $I$ denotes the identity,
\item if $A\leq A'$ and $B\leq B'$, then $M(A,B)\leq M(A',B')$,
\item $CM(A,B)C\leq M(CAC,CBC)$ for all Hermitian $C$,
\item if $A_n\downarrow A$ and $B_n\downarrow B$ then $M(A_n,B_n)\downarrow M(A,B)$,
\end{enumerate}
where $\downarrow$ denotes the convergence in the strong operator topology of a monotone decreasing net.
\end{definition}
In property (ii), (iii), (iv) the partial order being used is the positive definite order i.e., $A\leq B$ if and only if $B-A$ is positive semidefinite. An important consequence of these properties is \cite{kubo} that every operator mean can be uniquely represented by a positive, normalized, operator monotone function $f(t)$ in the following form
\begin{equation}\label{mean}
M(A,B)=A^{1/2}f\left(A^{-1/2}BA^{-1/2}\right)A^{1/2}\text{.}
\end{equation}
This unique $f(t)$ is said to be the representing function of the operator mean $M(A,B)$. So actually operator means are in one-to-one correspondence with normalized operator monotone functions, the above characterization provides an order-isomorphism between them. Normalization means that $f(1)=1$. For symmetric means i.e., for means $M(A,B)=M(B,A)$, we have $f(t)=tf(1/t)$ which implies that $f'(1)=1/2$. Operator monotone functions have strong continuity properties, namely all of them are analytic functions and can be analytically continued to the upper complex half-plane.

The set of all operator means is denoted by $\mathfrak{M}$ i.e.,
\begin{equation*}
\begin{split}
\mathfrak{M}=\{&M(\cdot,\cdot):M(A,B)=A^{1/2}f(A^{-1/2}BA^{-1/2})A^{1/2}, f:(0,\infty)\mapsto(0,\infty)\\
&\text{ operator monotone }, f(1)=1\}.
\end{split}
\end{equation*}
Similarly $\mathfrak{m}=\{f(x):f$ is a representing function of an $M\in\mathfrak{M}\}$. We will use the notation $\mathfrak{m}(t)$ to denote the set of all operator monotone functions $f$ on $(0,\infty)$ such that $f(x)>0$ for all $x\in(0,\infty)$ and $f(1)=1, f'(1)=t$. We can find the minimal and maximal elements of $\mathfrak{m}(t)$ for all $t\in(0,1)$ easily.
\begin{lemma}\label{maxmininp}
For all $f(x)\in\mathfrak{m}(t)$ we have
\begin{equation}
\left((1-t)+tx^{-1}\right)^{-1}\leq f(x)\leq (1-t)+tx.
\end{equation}
\end{lemma}
\begin{proof}
Since every operator monotone function is operator concave, see Chapter V \cite{bhatia}, therefore we must have $f(x)\leq (1-t)+tx$ by concavity and the normalization conditions on elements of $\mathfrak{m}(t)$. Since the map $x^{-1}$ is order reversing on positive matrices, we have that if $f(x)\in\mathfrak{m}(t)$ then also $f(x^{-1})^{-1}\in\mathfrak{m}(t)$, which leads to the lower bound.
\end{proof}

Since $\left((1-t)+tx^{-1}\right)^{-1}$ and $(1-t)+tx$ are operator monotone, they are the minimal and maximal elements of $\mathfrak{m}(t)$ respectively, and also they are the representing functions of the weighted harmonic and arithmetic means. In general by the previous Lemma~\ref{maxmininp} the inequality
\begin{equation}\label{eq:maxmininp}
\left[(1-t)A^{-1}+tB^{-1}\right]^{-1}\leq M(A,B)\leq (1-t)A+tB.
\end{equation}
is true for all $M(A,B)$ operator means with representing operator monotone function $f$ for which we have $f'(1)=t$. In this sense $\mathfrak{m}(t)$ characterizes weighted operator means. If we take this as the definition of weighted operator means, one can compare it with the definition of weighted matrix means given in \cite{palfia3}.

We will make use of the following characterization due to Hansen using the notation
$$h_s(x):=[(1-s)+sx^{-1}]^{-1}$$
for $s\in[0,1]$.

\begin{proposition}[Theorem 4.9 in \cite{hansen3}]\label{harmonicreprprop}
Let $f\in\mathfrak{m}$. Then
\begin{equation}\label{harmonicrepr}
f(x)=\int_{[0,1]}[(1-s)+sx^{-1}]^{-1}d\nu(s)
\end{equation}
where $\nu$ is a probability measure over the closed interval $[0,1]$.
\end{proposition}

There are two degenerate cases of operator means induced by a $\nu$ which are supported only over the single points ${0}$ or ${1}$. One of them is the left trivial mean
\begin{equation*}
l(x)=1
\end{equation*}
with represented operator mean $M(A,B)=A$ and the right trivial mean
\begin{equation*}
r(x)=x
\end{equation*}
with represented operator mean $M(A,B)=B$.

\begin{proposition}\label{positivemapthm}
Let $\Phi$ be a positive unital linear map and $M\in\mathfrak{M}$. Then
\begin{equation*}
\Phi(M(A,B))\leq M(\Phi(A),\Phi(B))
\end{equation*}
for $A,B>0$.
\end{proposition}
\begin{proof}
Using Proposition~\ref{harmonicreprprop} we have that
\begin{equation*}
M(A,B)=\int_{[0,1]}[(1-s)A^{-1}+sB^{-1}]^{-1}d\nu(s)
\end{equation*}
where $\nu$ is a probability measure on $[0,1]$. By Theorem 4.1.5 in \cite{bhatia2} we have that
\begin{equation*}
\Phi([(1-s)A^{-1}+sB^{-1}]^{-1})\leq [(1-s)\Phi(A)^{-1}+s\Phi(B)^{-1}]^{-1}.
\end{equation*}
Using the fact that $\nu$ can be approximated by finitely supported measures and the linearity of $\Phi$, we get from the above that
\begin{equation*}
\begin{split}
\Phi(\int_{[0,1]}[(1-s)A^{-1}+sB^{-1}]^{-1}d\nu(s))=\int_{[0,1]}\Phi([(1-s)A^{-1}+sB^{-1}]^{-1})d\nu(s)\\
\leq \int_{[0,1]}[(1-s)\Phi(A)^{-1}+s\Phi(B)^{-1}]^{-1}d\nu(s).
\end{split}
\end{equation*}

\end{proof}

In \cite{kubo} Kubo and Ando defined the transpose of a matrix mean $M(A,B)$ as
\begin{equation}\label{transpose}
M'(A,B)=M(B,A).
\end{equation}
By Proposition~\ref{harmonicreprprop} it is clear that for an
\begin{equation*}
M(A,B)=\int_{[0,1]}[(1-s)A^{-1}+sB^{-1}]^{-1}d\nu(s)
\end{equation*}
we have that
\begin{equation*}
M'(A,B)=M(B,A)=\int_{[0,1]}[(1-s)B^{-1}+sA^{-1}]^{-1}d\nu(s).
\end{equation*}
So if $M'(A,B)$ has corresponding measure $\nu'$, then $d\nu'(s)=d\nu(1-s)$. Similarly for the representing functions we have $f^{tr}(x)=xf(1/x)$. Also symmetric means $M(A,B)=M'(A,B)$ have corresponding probability measures $\nu$ such that $d\nu(s)=d\nu(1-s)$ and vice versa.

We have one more result characterizing the partial order between operator means which should be well known, however its proof cannot be found anywhere in the literature.

\begin{proposition}\label{P:ordermean}
Let $M,N\in\mathfrak{M}$ with representing functions $f,g\in\mathfrak{m}$ respectively. Then
$M(A,B)\leq N(A,B)$ for all $A,B\in\mathbb{P}$ if and only if $f(x)\leq g(x)$ for $x>0$.
\end{proposition}
\begin{proof}
From the definition of the transpose mean we have the following formulas
\begin{eqnarray*}
M(A,B)=A^{1/2}f(A^{-1/2}BA^{-1/2})A^{1/2}=M'(B,A)=B^{1/2}f^{tr}(B^{-1/2}AB^{-1/2})B^{1/2}\\
N(A,B)=A^{1/2}g(A^{-1/2}BA^{-1/2})A^{1/2}=N'(B,A)=B^{1/2}g^{tr}(B^{-1/2}AB^{-1/2})B^{1/2}
\end{eqnarray*}
where $f^{tr}(x)=xf(1/x)$ and $g^{tr}(x)=xg(1/x)$.

So now suppose $f(x)\leq g(x)$ for $x>0$. Then it follows that also $f^{tr}(x)\leq g^{tr}(x)$, hence by the above correspondence $M(A,B)\leq N(A,B)$. The reverse implication also follows from the reverse of the same argument.
\end{proof}

\section{Contraction property of operator means}
In the section we prove further properties of operator means using explicitly the integral characterizations given in Proposition~\ref{harmonicreprprop}. We will use the results given in this section to generalize the construction of matrix power means to all possible operator means in later sections.

Let $E$ be a Hilbert space, $\mathfrak{B}(E)$ denote the Banach space of bounded linear operators, $S(E)$ denote the Banach space of bounded linear self-adjoint operators and $\mathbb{P}\subseteq S(E)$ the cone of positive definite and $\mathbb{P}_0$ the cone of positive semi-definite operators. On $\mathbb{P}$ we have the positive definite order similarly to the finite dimensional case which means that $A\leq B$ for $A,B\in\mathbb{P}$ if and only if $B-A\in\mathbb{P}_0$. It is also easy to see that if for $A,B\in S(E)$ and $0\leq A\leq B$ then also $\left\|A\right\|\leq \left\|B\right\|$ \cite{lawsonlim1}. We will use the notation $[A,B]$ for the order interval generated by $A\leq B$ i.e., $[A,B]=\{X\in\mathbb{P}:A\leq X\leq B\}$. We also have that $\mathbb{P}=\bigcup_{k=1}^{\infty}\left[\frac{1}{k}I,kI\right]$.

We would also like to consider measures taking values in $\mathbb{P}$ and we would like to integrate with respect to these measures. There are a number of different (stronger or weaker) ways to do that. In our setting the integral that we need is the weak operator form of the Pettis integral. First let us introduce weak operator measurability. 
\begin{definition}[weak measurability, weak integrability]
Let $(\Omega,\mathcal{A},\mu)$ be a finite measure space. Let $f:\Omega\to S(E)$ be given. The function $f$ is said to be weakly (operator) measurable if and only if
$$\left\langle f(\omega)x,y\right\rangle$$
is $\mu$ measurable for all $x,y\in E$.

We also say that the function $f$ is weakly integrable if it is weakly measurable and
$$\int_\Omega|\left\langle f(\omega)x,y\right\rangle| d\mu(\omega)<+\infty$$
for all $x,y\in E$.
\end{definition}

Now integrability is defined as follows.
\begin{definition}[weak operator Pettis integral]\label{D:pettis}
Let $(\Omega,\mathcal{A},\mu)$ be a finite measure space. Let $f:\Omega\to S(E)$ be weakly operator measurable. If there exists $A\in S(E)$ such that $\left\langle Ax,y\right\rangle=\int_{\Omega}\left\langle f(\omega)x,y\right\rangle d\mu(\omega)$ for all $x,y\in E$, then we define
$$\int_{\Omega}fd\mu:=A.$$
\end{definition}
Clearly the uniqueness of the integral is satisfied. The above definition is based on the definition of the Pettis integral in an arbitrary Banach space $B$, where we require the existence of $x\in B$, for an $f:\Omega\mapsto B$ such that $\left\langle x,y^{*}\right\rangle=\int_{\Omega}\left\langle f(\omega),y^{*}\right\rangle d\mu(\omega)$ where $y^{*}\in B^{*}$. In other words first we require the weak measurability of $f$ i.e., $\left\langle f(\omega),y^{*}\right\rangle$ is measurable for all $y^{*}\in B^{*}$, and then integrability i.e., the existence of such $x\in B$.
Now there is a different notion of integrability due to Dunford which asks for the existence of $x^{**}\in B^{**}$ such that $\left\langle x^{**},y^{*}\right\rangle=\int_{\Omega}\left\langle f(\omega),y^{*}\right\rangle d\mu(\omega)$ for all $y^{*}\in B^{*}$. It is known that every weakly measurable function $f$ that is weakly integrable (i.e., $\int_{\Omega}|\left\langle f(\omega),y^{*}\right\rangle| d\mu(\omega)<+\infty$ for all $y^{*}\in B^{*}$) is Dunford integrable, see Theorem 11.55 in \cite{aliprantis}. In order to ensure the existence of the operator $A\in S(E)$ in Definition~\ref{D:pettis} we need some preliminary results.
\begin{lemma}\label{L:weakoperatorPettis}
Let $(\Omega,\mathcal{A},\mu)$ be a finite measure space and let $f:\Omega\to S(E)$ be weakly measurable and weakly integrable. Then the weak operator Pettis integral $\int_{\Omega}fd\mu$ exists.
\end{lemma}
\begin{proof}
By assumption $\left\langle f(\omega)x,y\right\rangle$ is $\mu$-measurable and $\int_{\Omega}|\left\langle f(\omega)x,y\right\rangle| d\mu(\omega)<+\infty$ for all $x,y\in E$. Define $F_x:\Omega\lra E$ by $F_x(\omega):=f(\omega)x$. Then we have for fixed $x\in E$ that for all $y\in E$, $\left\langle F_x(\omega),y\right\rangle$ is $\mu$-measurable and $\int_{\Omega}|\left\langle F_x(\omega),y\right\rangle| d\mu(\omega)<+\infty$, and since $E$ is Hilbert therefore the linear functionals $y^*(\cdot):=\left\langle \cdot,y\right\rangle$ exhaust $E^*$ by the Riesz representation theorem. Therefore the function $F_x:\Omega\lra E$ is Dunford integrable i.e., there exists a unique $\hat{F}(x)\in E^{**}\simeq E$ such that $\int_{\Omega}\left\langle F_x(\omega),y\right\rangle d\mu(\omega)=\left\langle \hat{F}(x),y\right\rangle$. Now it is routine to check that $x\mapsto\hat{F}(x)$ defines a bounded (it maps bounded sets to bounded sets) linear operator on $E$, moreover by the assumptions $\hat{F}\in S(E)$.
\end{proof}


We have a generalized version of the Dominated Convergence Theorem for the Pettis integral that makes us of:
\begin{theorem}\label{T:domconv}
Let $f:\Omega\mapsto S(E)$ satisfy the following:
\begin{itemize}
\item There exists a sequence of weak operator Pettis integrable functions $f_n$ such that $\lim_{n\to\infty}\left\langle f_nx,y\right\rangle=\left\langle fx,y\right\rangle$ in measure for all $x,y\in E$.
\item There exists a real valued $\mu$-integrable function $h:\Omega\to\mathbb{R}$ such that for each linear functional $l(Z):=\left\langle Zx,y\right\rangle$ in the weak*-compact polar of some neighborhood generated by the duality provided by the weak operator topology and $n\in\mathbb{N}$ the inequality $|\left\langle f_nx,y\right\rangle|\leq h$ holds a.e.
\end{itemize}
Then $f$ is weak operator Pettis integrable and
$$\lim_{n\to\infty}\int f_nd\mu=\int fd\mu$$
in the weak operator topology.
\end{theorem}
This theorem is a consequence of the Lebesgue Dominated Convergence Theorem and the Banach space form of it can be found in \cite{musial}.

It is well known that all Bochner integrable functions with respect to a measure are Pettis integrable. It is also known that a Bochner integrable function always has separable range. That is one reason why we cannot apply the Bochner integral, since $S(E)$ contains non-separable subspaces. Also if $\Omega$ is a compact Hausdorff space with a probability measure $\mu$ and $f:\Omega\mapsto X$ is continuous where $X$ is a Banach space then $f$ is Pettis integrable with respect to $\mu$, see Theorem 3.27 in \cite{rudin}.

On $\mathbb{P}$ the partial order induces a complete metric space structure \cite{thompson}. The Thompson or part metric is defined as
\begin{equation}\label{thompsonmetric}
d_\infty(A,B)=\max\left\{\log M(A/B),M(B/A)\right\}
\end{equation}
for any $A,B\in\mathbb{P}$, where $M(A/B)=\inf\{\alpha:A\leq\alpha B\}$. The metric space $(\mathbb{P},d_\infty)$ is complete and has some several other nice properties.
\begin{lemma}[Lemma 10.1 in \cite{lawsonlim0}]\label{thompsonproperties}
We have
\begin{enumerate}
	\item $d_\infty(rA,rB)=d_\infty(A,B)$ for any $r>0$,
	\item $d_\infty(A^{-1},B^{-1})=d_\infty(A,B)$,
	\item $d_\infty(MAM^*,MBM^*)=d_\infty(A,B)$ for all $M\in \mathrm{GL}(E)$ where $\mathrm{GL}(E)$ denotes the Banach-Lie group of all invertible bounded linear operators on $E$,
	\item $d_\infty(\sum_{i=1}^kt_iA_i,\sum_{i=1}^kt_iB_i)\leq \max_{1\leq i\leq k}d_\infty(A_i,B_i)$ where $t_i>0$,
	\item $e^{-d_\infty(A,B)}B\leq A\leq e^{d_\infty(A,B)}B$ and $e^{-d_\infty(A,B)}A\leq B\leq e^{d_\infty(A,B)}A$.
\end{enumerate}
\end{lemma}

Property 4. in Lemma~\ref{thompsonproperties} is important for us, but we need a refined and also a weighted version of it. The refined version considers integrals instead of sums.

\begin{lemma}\label{L:infinite4}
Let $(\Omega,\mathcal{A},\mu)$ be a probability space, and let $f_1,f_2:\Omega\mapsto\mathbb{P}$ be measurable such that both $f_1$ and $f_2$ are weak operator Pettis integrable with respect to $\mu$. Then
\begin{enumerate}
	\item[(4')] $d_\infty(\int_{\Omega}f_1d\mu,\int_{\Omega}f_2d\mu)\leq \sup_{\omega\in\supp(\mu)}d_\infty(f_1(\omega),f_2(\omega))$, assuming that the supremum is finite.
\end{enumerate}
\end{lemma}
\begin{proof}
Let $\alpha:=e^{\sup_{\omega\in\supp(\mu)}d_\infty(f_1(\omega),f_2(\omega))}$. Then for all $\omega\in\supp(\mu)$ we have 
\begin{eqnarray*}
f_1(\omega)&\leq& \alpha f_2(\omega)\\
f_2(\omega)&\leq& \alpha f_1(\omega).
\end{eqnarray*}
Now $f_1(\omega)\leq \alpha f_2(\omega)$ is equivalent to
$$\left\langle [\alpha f_2(\omega)-f_1(\omega)]x,x\right\rangle\geq 0$$
for all $x\in E$. Then by the definition of the weak operator Pettis integral we have
\begin{eqnarray*}
\int_{\Omega}\left\langle \left[\alpha f_2(\omega)-f_1(\omega)\right]x,x\right\rangle d\mu&\geq& 0\\
\left\langle \left\{\int_{\Omega}\left[\alpha f_2(\omega)-f_1(\omega)\right]d\mu\right\} x,x\right\rangle &\geq& 0\\
\left\langle \left[\alpha \int_{\Omega}f_2(\omega) d\mu-\int_{\Omega}f_1(\omega) d\mu\right] x,x\right\rangle &\geq& 0\\
\int_{\Omega}f_1 d\mu&\leq& \alpha \int_{\Omega}f_2 d\mu.
\end{eqnarray*}
Now similar computation leads to $\int_{\Omega}f_2 d\mu\leq \alpha \int_{\Omega}f_1 d\mu$ starting from the other inequality.
\end{proof}

Here is the weighted version of property (4) for finite sums.

\begin{proposition}\label{weightedthompson}
Let Let $c_i>0$ be real numbers, $A_i,B_i\in\mathbb{P}$, $1\leq i\leq 2$ and suppose that $d_\infty(A_1,B_1)\geq d_\infty(A_2,B_2)$. Then we have
\begin{equation*}
\begin{split}
e^{d_{\infty}\left(c_1A_1+c_2A_2,c_1B_1+c_2B_2\right)}\leq&\max\left\{\frac{c_1e^{d_\infty(A_1,B_1)}+c_2e^{-d_{\infty}(A_1,A_2)}e^{d_\infty(A_2,B_2)}}{c_1+c_2e^{-d_{\infty}(A_1,A_2)}},\right.\\
&\left.\frac{c_1e^{d_\infty(A_1,B_1)}+c_2e^{-d_{\infty}(B_1,B_2)}e^{d_\infty(A_2,B_2)}}{c_1+c_2e^{-d_{\infty}(B_1,B_2)}}\right\}.
\end{split}
\end{equation*}
\end{proposition}
\begin{proof}
Let $\alpha_i=e^{d_\infty(A_i,B_i)}$. By definition of $d_{\infty}$, in order to calculate $d_{\infty}\left(\sum_{i=1}^2c_iA_i,\sum_{i=1}^2c_iB_i\right)$, we are looking for the infimum of all $\beta\geq 0$ such that both
\begin{equation}\label{weightedthompson:eq2}
\begin{split}
c_1A_1+c_2A_2\leq \beta(c_1B_1+c_2B_2),\\
c_1B_1+c_2B_2\leq \beta(c_1A_1+c_2A_2)
\end{split}
\end{equation}
are satisfied. We also have that 
\begin{equation*}
\begin{split}
c_1A_1+c_2A_2\leq c_1\alpha_1B_1+c_2\alpha_2B_2\\
c_1B_1+c_2B_2\leq c_1\alpha_1A_1+c_2\alpha_2A_2.
\end{split}
\end{equation*}
From this it follows that the infimum of the $\beta$ satisfying \eqref{weightedthompson:eq2} are bounded above by the infimum of $\beta\geq 0$ satisfying
\begin{equation}\label{weightedthompson:eq1}
\begin{split}
c_1A_1+c_2A_2\leq c_1\alpha_1B_1+c_2\alpha_2B_2\leq \beta(c_1B_1+c_2B_2),\\
c_1B_1+c_2B_2\leq c_1\alpha_1A_1+c_2\alpha_2A_2\leq \beta(c_1A_1+c_2A_2).
\end{split}
\end{equation}
The first inequality above is equivalent to
\begin{equation*}
0\leq \sum_{i=1}^2(\beta-\alpha_i)c_iB_i.
\end{equation*}
Now by Property 4. in Lemma~\ref{thompsonproperties} we have the natural bound $\beta\leq \max_{1\leq i\leq 2}\alpha_i$. So we may try to find a better bound by assuming that $\alpha_1\geq \beta\geq \alpha_2$, 
where without loss of generality $\alpha_1\geq \alpha_2$. 
Using the assumption on $\beta$ we get that we seek the infimum of all $\beta\geq 0$ such that
\begin{equation*}
(\beta-\alpha_2)c_2B_2\geq (\alpha_1-\beta)c_1B_1.
\end{equation*}
Using property 5. in Lemma~\ref{thompsonproperties} we have that the sought $\beta$ above is bounded above by the infimum of all $\beta\geq 0$ satisfying
\begin{equation*}
(\beta-\alpha_2)c_2e^{-d_{\infty}(B_1,B_2)}\geq (\alpha_1-\beta)c_1.
\end{equation*}
This is equivalent to
\begin{equation*}
\sum_{i=1}^2(\beta-\alpha_i)c_ie^{-d_{\infty}(B_1,B_i)}\geq 0,
\end{equation*}
in other words we have that
\begin{equation*}
\beta\geq \frac{\sum_{i=1}^2c_i\alpha_ie^{-d_\infty(B_1,B_i)}}{\sum_{i=1}^2c_ie^{-d_\infty(B_1,B_i)}}=\frac{\sum_{i=1}^2e^{d_\infty(A_i,B_i)}c_ie^{-d_\infty(B_1,B_i)}}{\sum_{i=1}^2c_ie^{-d_\infty(B_1,B_i)}}.
\end{equation*}
Performing the same calculation ($A_i$ in place of $B_i$) by starting with the second inequality in \eqref{weightedthompson:eq1} we get
\begin{equation*}
\beta\geq \frac{\sum_{i=1}^2e^{d_\infty(A_i,B_i)}c_ie^{-d_\infty(A_1,A_i)}}{\sum_{i=1}^2c_ie^{-d_\infty(A_1,A_i)}}.
\end{equation*}
This means that
\begin{equation*}
\begin{split}
e^{d_{\infty}\left(\sum_{i=1}^2c_iA_i,\sum_{i=1}^2c_iB_i\right)}\leq&\max\left\{\frac{\sum_{i=1}^2e^{d_\infty(A_i,B_i)}c_ie^{-d_{\infty}(A_1,A_i)}}{\sum_{i=1}^2c_ie^{-d_{\infty}(A_1,A_i)}},\right.\\
&\left.\frac{\sum_{i=1}^2e^{d_\infty(A_i,B_i)}c_ie^{-d_{\infty}(B_1,B_i)}}{\sum_{i=1}^2c_ie^{-d_{\infty}(B_1,B_i)}}\right\}.
\end{split}
\end{equation*}
which is what we wanted to prove.
\end{proof}

Let $\overline{B}_A(r)=\{X\in\mathbb{P}:d_{\infty}(A,X)\leq r\}$.
The following elementary inequality for the weighted arithmetic mean will be useful. Let $\Delta_{n}$ denote the convex set of positive probability n-vectors i.e., $\omega=(w_{1},\dots,w_{n})\in\Delta_{n}$ if $w_i>0$ and $\sum_{i=1}^nw_i=1$.
\begin{lemma}\label{L:elementaryarithm}
Let $a_i\geq 0$, $1\leq i\leq n$, and suppose that $a_1\geq a_i$, for all $1\leq i\leq n$. Let $\omega=(w_{1},\dots,w_{n})\in\Delta_n$ and let $\mu=(u_{1},\dots,u_{n})\in\Delta_n$ such that $w_1\geq u_1$, and $w_i\leq u_i$ for $2\leq i\leq n$. Then we have
\begin{equation}
\sum_{i=1}^nw_ia_i\geq \sum_{i=1}^nu_ia_i.\tag{*}
\end{equation}
\end{lemma}
\begin{proof}
We have
\begin{eqnarray*}
\sum_{i=1}^nw_ia_i-\sum_{i=1}^nu_ia_i&=&(w_1-u_1)a_1-\sum_{i=2}^n(u_i-w_i)a_i\\
&\geq& (w_1-u_1)a_1-\sum_{i=2}^n(u_i-w_i)a_1\\
&=&\sum_{i=1}^nw_ia_1-\sum_{i=1}^nu_ia_1=0.
\end{eqnarray*}
\end{proof}

A cone $C$ is \emph{almost Archimedean} if the closure of the intersection of the cone with any two-dimensional subspace is still a cone, that is, contains no nontrivial subspace. Note that $\mathbb{P}$ is almost Archimedean, so we can use the following contraction result for the vector addition in $\mathbb{P}$.

We will need a basic contraction result for the weighted arithmetic and harmonic means on the cone $\mathbb{P}$. The following contraction result has already appeared in \cite{lawsonlim3}.
\begin{theorem}[Theorem 2.6 \cite{lawsonlim3}]
Let $C$ be an almost Archimedean cone in a real vector space $V$ such
that $C$ does not contain 0 and consists of only one part. Let $a,b\in C$. Then with respect to the Thompson metric the translation $\tau_a(x):=a+x$ restricted to $\{x\in C:x\leq b\}$ has Lipschitz constant $\frac{M(b/a)}{M(b/a)+1}$.
\end{theorem}
Since $\mathbb{P}$ is almost Archimedean the above is applicable, however we state here our own contraction result below. Even though our Lipschitz constant is not as sharp as in the above, its proof will serve as a reference point for more involved calculations to come.

\begin{lemma}\label{arithmeticcontraction}
Let $a,b>0$ be real numbers. Then the mappings $h^+_{a,b,A}(B)=aA+bB$ and $h^-_{a,b,A}(B)=(aA^{-1}+bB^{-1})^{-1}$ are strict contractions on every $\overline{B}_A(r)$ for all $r<\infty$ i.e., for all $X,Y\in\overline{B}_A(r)$
\begin{equation*}
d_\infty(h^\pm_{a,b,A}(X),h^\pm_{a,b,A}(Y))\leq \rho d_\infty(X,Y)
\end{equation*}
where
\begin{equation*}
\rho=\frac{\log\frac{be^{3r}+a}{be^{r}+a}}{2r}
\end{equation*}
for a fixed $0<r<\infty$.
\end{lemma}
\begin{proof}
It suffices to prove the above for $h^+_{a,b,A}(B)$, since then the same follows for $h^-_{a,b,A}(B)$ by the inversion invariancy of the metric $d_\infty$: property (2) in Lemma~\ref{thompsonproperties}. Also by property (3) in Lemma~\ref{thompsonproperties} it is enough to prove for the case when $A=I$. Let $X,Y\in\overline{B}_I(r)$. By Proposition~\ref{weightedthompson} we have that
\begin{equation*}
\begin{split}
e^{d_\infty(h^+_{a,b,I}(X),h^+_{a,b,I}(Y))}\leq\max&\left\{\frac{ae^{-d_\infty(X,I)}+be^{d_\infty(X,Y)}}{ae^{-d_\infty(X,I)}+b},\right.\\
&\left.\frac{ae^{-d_\infty(Y,I)}+be^{d_\infty(X,Y)}}{ae^{-d_\infty(Y,I)}+b}\right\}
\end{split}
\end{equation*}
Since $X,Y\in\overline{B}_A(r)$ we have that
\begin{equation*}
e^{-r}I\leq X,Y\leq e^rI
\end{equation*}
which means that
\begin{equation*}
d_\infty(X,I),d_\infty(Y,I)\leq r.
\end{equation*}
Hence it follows that
\begin{equation*}
\frac{ae^{-d_\infty(Y,I)}}{ae^{-d_\infty(Y,I)}+b}=\frac{1}{1+e^{d_\infty(Y,I)}b/a}\geq \frac{1}{1+e^{r}b/a}=\frac{ae^{-r}}{ae^{-r}+b}
\end{equation*}
and similarly
\begin{equation*}
\frac{ae^{-d_\infty(X,I)}}{ae^{-d_\infty(X,I)}+b}\geq \frac{ae^{-r}}{ae^{-r}+b}.
\end{equation*}
Also
\begin{eqnarray*}
\frac{b}{ae^{-d_\infty(X,I)}+b}\leq \frac{b}{ae^{-r}+b}\\
\frac{b}{ae^{-d_\infty(Y,I)}+b}\leq \frac{b}{ae^{-r}+b},
\end{eqnarray*}
hence using (*) this means that
\begin{equation}\label{eq:usingstar}
e^{d_\infty(h^+_{a,b,I}(X),h^+_{a,b,I}(Y))}\leq\frac{ae^{-r}+be^{d_\infty(X,Y)}}{ae^{-r}+b}.
\end{equation}
So we seek $0<\rho<1$ such that
\begin{equation*}
d_\infty(h^+_{a,b,I}(X),h^+_{a,b,I}(Y))\leq\log\left(\frac{ae^{-r}+be^{d_\infty(X,Y)}}{ae^{-r}+b}\right)\leq \rho d_\infty(X,Y)
\end{equation*}
for all $X,Y\in\overline{B}_I(r)$ i.e., $d_\infty(X,Y)\leq 2r$. It therefore suffices to find the maximum of the function
\begin{equation*}
f(x)=\frac{\log\left(\frac{e^{x}be^{r}+a}{be^{r}+a}\right)}{x}
\end{equation*}
on the interval $[0,2r]$. First routine calculations show that 
\begin{equation*}
\lim_{x\to 0+}\frac{\log\left(\frac{e^{x}be^{r}+a}{be^{r}+a}\right)}{x}=\frac{be^{r}}{be^{r}+a}<1.
\end{equation*}
Then the maximization problem is the same as finding the smallest $\rho<1$ such that
\begin{equation*}
\log\left(\frac{xbe^{r}+a}{be^{r}+a}\right)\leq\rho\log(x)
\end{equation*}
on the transformed interval $[1,e^{2r}]$. But that is equivalent to
\begin{equation*}
\frac{xbe^{r}}{be^{r}+a}+\frac{a}{be^{r}+a}\leq x^{\rho}.
\end{equation*}
Since for $0<\rho<1$ the function $x^{\rho}$ is concave monotonically increasing, $1^{\rho}=1$ and its derivative is $\rho$ at $1$ therefore the smallest such $\rho$ is determined at the $e^{2r}$ endpoint of the interval, hence
\begin{equation*}
\rho=\frac{\log\left(\frac{e^{2r}be^{r}+a}{be^{r}+a}\right)}{2r}.
\end{equation*}
\end{proof}

\begin{remark}
By Proposition~\ref{weightedthompson} it is clear that the functions $h^\pm_{a,b,A}(X)$ for all $a,b\geq 0$ are nonexpansive on the whole $\mathbb{P}$ i.e.,
\begin{equation*}
d_\infty(h^\pm_{a,b,A}(X),h^\pm_{a,b,A}(Y))\leq d_\infty(X,Y).
\end{equation*}
\end{remark}

%

\begin{remark}
By Lemma~\ref{arithmeticcontraction} it follows that the weighted arithmetic $f(B):=(1-s)A+sB$ and harmonic $g(B):=((1-s)A^{-1}+sB^{-1})^{-1}$ means are strict contractions on $\overline{B}_A(r)$ for all $r<\infty$ and $s\in[0,1)$. The contraction coefficients $\rho$ are striclty less than $1$ for all $s\in(0,1)$, in general the $\rho$ calculated explicitly in the proof of Lemma~\ref{arithmeticcontraction} is a monotonically decreasing function in $s$ and $\rho\to 1$ as $s\to 1-$ also $\rho\to 0$ as $s\to 0+$. The case $s=1$ gives the right trivial mean $M(A,B)=B$ that is nonexpansive i.e.,
\begin{equation*}
d_\infty(M(A,X),M(A,Y))\leq d_\infty(X,Y),
\end{equation*}
while $s=0$ is the left trivial mean $M(A,B)=A$ and it has contraction coefficient $0$ on all of $\mathbb{P}$.
\end{remark}

These preliminary results yield the following contraction result.

\begin{theorem}\label{contract}
Let $M\in\mathfrak{M}$ and $f(X)=M(A,X)$. If $M$ is not the right trivial mean (i.e., $M(A,B)\neq B$) then the mapping $f(X)$ is a strict contraction on $\overline{B}_A(r)$ for all $r<\infty$ i.e., there exists $0<\rho_r<1$ such that
\begin{equation*}
d_\infty(f(X),f(Y))\leq\rho_r d_\infty(X,Y)
\end{equation*}
for all $X,Y\in\overline{B}_A(r)$.

If $M$ is the right trivial mean (i.e., $M(A,B)=B$) then $f(X)$ is nonexpansive on $\mathbb{P}$, that is
\begin{equation*}
d_\infty(f(X),f(Y))\leq d_\infty(X,Y)
\end{equation*}
for all $A,X,Y\in\mathbb{P}$.
\end{theorem}
\begin{proof}
The case of the right trivial mean is just the preceding remark, so assume that $M(A,B)$ is not the right trivial mean. Again by property 3 in Lemma~\ref{thompsonproperties} it is enough to prove for the case when $A=I$. By Proposition~\ref{harmonicreprprop} the mean $M\in\mathfrak{M}$ is represented as
\begin{equation}\label{contractint}
M(I,X)=\int_{[0,1]}[(1-s)I+sX^{-1}]^{-1}d\nu(s)=\int_{[0,1]}h_s(X)d\nu(s)
\end{equation}
So let $X,Y\in\overline{B}_I(r)$. There are other simple cases when the probability measure is supported only over the two points $\{0\},\{1\}$. These cases include the weighted arithmetic mean with $s\in(0,1)$ and the case of the left trivial mean which is covered in the preceding remark and are clearly strict contractions on $\overline{B}_I(r)$.

For the remaining cases we split the integral in \eqref{contractint} to the sum of integrals over the mutually disjoint intervals $I_1=[0,a]$, $I_2=(a,1]$ for some $a\in(0,1)$ such that $\nu$ has nonzero mass on the interval $I_1$. Such an $a$ clearly exists since we have excluded the case when the measure $\nu$ is supported only on the point $\{1\}$. We have that
\begin{equation*}
f_{i}(X)=\int_{I_i}[(1-s)I+sX^{-1}]^{-1}d\nu(s).
\end{equation*}
Notice that the integrals with respect to the measures defining $f_i$ are finite positive measures, so with appropriate rescaling we can consider the integration with respect to a probability measure times a constant. Also the interval $[0,1]$ is a compact Hausdorff space and the function $h_s(X)$ mapping to $\mathbb{P}$ is continuous with respect to $s\in[0,1]$, so Theorem 3.27 in \cite{rudin} applies and $h_s(X)$ is integrable and therefore property (4') also applies for the metric $d_\infty$. So the functions $f_{1}$ and $f_2$ are nonexpansive due to the preceding remarks i.e., we have
\begin{equation*}
d_\infty(f_{i}(X),f_{i}(Y))\leq \rho_i d_\infty(X,Y)
\end{equation*}
where $\rho_2=1$. Again due to property (4'), the preceding remarks and Lemma~\ref{arithmeticcontraction} we have that
\begin{equation*}
\rho_1=\frac{\log\left(\frac{e^{2r}be^{r}+a}{be^{r}+a}\right)}{2r}
\end{equation*}
since it is easy to see that by Lemma~\ref{arithmeticcontraction} the contraction coefficient $\rho$ calculated in Lemma~\ref{arithmeticcontraction} corresponding to a $h_s(X)$ with $s\in[0,a]$ is bounded from above by the contraction coefficient $\rho$ calculated in Lemma~\ref{arithmeticcontraction} corresponding to $h_a(X)$.

Now by Proposition~\ref{weightedthompson} we have
\begin{equation*}
\begin{split}
e^{d_\infty(f(X),f(Y))}&\leq\max\left\{\frac{\sum_{i=1}^{2}e^{d_\infty(f_i(X),f_i(Y))}e^{-d_\infty(f_m(X),f_i(X))}}{\sum_{i=1}^{2}e^{-d_\infty(f_m(X),f_i(X))}},\right.\\
&\left.\frac{\sum_{i=1}^{2}e^{d_\infty(f_i(X),f_i(Y))}e^{-d_\infty(f_m(Y),f_i(Y))}}{\sum_{i=1}^{2}e^{-d_\infty(f_m(Y),f_i(Y))}}\right\}\\
&\leq\max\left\{\frac{\sum_{i=1}^{2}e^{\rho_id_\infty(X,Y)}e^{-d_\infty(f_m(X),f_i(X))}}{\sum_{i=1}^{2}e^{-d_\infty(f_m(X),f_i(X))}},\right.\\
&\left.\frac{\sum_{i=1}^{2}e^{\rho_id_\infty(X,Y)}e^{-d_\infty(f_m(Y),f_i(Y))}}{\sum_{i=1}^{2}e^{-d_\infty(f_m(Y),f_i(Y))}}\right\}
\end{split}
\end{equation*}
where $m$ is such that $d_\infty(f_m(X),f_m(Y))\geq d_\infty(f_i(X),f_i(Y))$, $i=1,2$. To obtain the second inequality above we used the monotonicity of the functions $e^x$ and the weighted arithmetic mean with weights of the form $w_i=e^{-d_\infty(f_m(Y),f_i(Y))}$. The next step is to see that $d_\infty(f_j(Y),f_i(Y))$ is bounded for $i,j=1,2$. We have that $Y\in\overline{B}_I(r)$, hence we have
\begin{eqnarray*}
e^{-r}I\leq Y\leq e^rI
\end{eqnarray*}
and by the monotonicity of the harmonic mean this yields
\begin{eqnarray*}
e^{-r}I\leq [(1-s)I+sY^{-1}]^{-1}\leq e^rI.
\end{eqnarray*}
Integrating this we get
\begin{eqnarray*}
\int_{I_i}d\nu(s)e^{-r}I\leq f_i(Y)\leq \int_{I_i}d\nu(s)e^rI
\end{eqnarray*}
for $i=1,2.$ From this it follows that
\begin{equation*}
\max_{i,j=1,2}\sup_{X\in\overline{B}_I(r)}d_\infty(f_j(X),f_i(X))\leq L
\end{equation*}
where
$$L=2r+\left|\log\left(\int_{I_1}d\nu(s)\right)-\log\left(\int_{I_2}d\nu(s)\right)\right|$$
and clearly $L<\infty$ since $\int_{I_i}d\nu(s)>0$. Now since $\rho_2=1$ we have that
\begin{equation*}
\begin{split}
e^{d_\infty(f(X),f(Y))}&\leq\max\left\{\frac{\sum_{i=1}^{2}e^{\rho_id_\infty(X,Y)}e^{-d_\infty(f_m(X),f_i(X))}}{\sum_{i=1}^{2}e^{-d_\infty(f_m(X),f_i(X))}},\right.\\
&\left.\frac{\sum_{i=1}^{2}e^{\rho_id_\infty(X,Y)}e^{-d_\infty(f_m(Y),f_i(Y))}}{\sum_{i=1}^{2}e^{-d_\infty(f_m(Y),f_i(Y))}}\right\}\\
&\overset{(*)}{\leq}\frac{e^{d_\infty(X,Y)}+e^{\rho_1d_\infty(X,Y)}e^{-L}}{e^{-L}+1},
\end{split}
\end{equation*}
where to get the last inequality we used a similar argument as we did to obtain \eqref{eq:usingstar}. Similarly as in the end of the proof of Lemma~\ref{arithmeticcontraction} we seek some $0<\rho<1$ such that
\begin{equation*}
\log\left(\frac{e^{d_\infty(X,Y)}+e^{\rho_1d_\infty(X,Y)}e^{-L}}{e^{-L}+1}\right)\leq \rho d_\infty(X,Y)
\end{equation*}
for all $X,Y\in\overline{B}_I(r)$ i.e., $d_\infty(X,Y)\leq 2r$. By the same argument as in the end of the proof of Lemma~\ref{arithmeticcontraction} we see that
\begin{equation*}
\rho=\frac{\log\left(\frac{e^{2r}+e^{\rho_12r}e^{-L}}{e^{-L}+1}\right)}{2r}
\end{equation*}
suffices and clearly $\rho<1$.

\end{proof}

\begin{remark}
In \cite{lawsonlim0} Lawson and Lim provided an extension of the geometric, logarithmic and some other iterated means to several variables over $\mathbb{P}$ relying on the Ando-Li-Mathias construction provided in \cite{ando}. They established the above contractive property for these means. Our Theorem~\ref{contract} shows that in fact the construction is applicable to all operator means due to the contraction result Theorem~\ref{contract}, hence providing multivariable extensions which work in the possibly infinite dimensional setting of $\mathbb{P}$. This was only known in the finite dimensional setting so far which case was proved in \cite{palfia3}.
\end{remark}

The further importance of Theorem~\ref{contract} will be apparent in the following sections, when we consider matrix (in fact operator) equations similarly to the case of the matrix power means. We close the section with a general nonexpansive property.
\begin{proposition}\label{nonexpansivemean}
Let $M:\mathbb{P}^k\to \mathbb{P}$ be such that
\begin{enumerate}
	\item if $A_i\leq B_i$ for all $1\leq i\leq k$, then $M(A_1,\ldots,A_k)\leq M(B_1,\ldots,B_k)$,
	\item if $t>0$, then $M(tA_1,\ldots,tA_k)=tM(A_1,\ldots,A_k)$,
\end{enumerate}
then
\begin{equation*}
d_\infty(M(A_1,\ldots,A_k),M(B_1,\ldots,B_k))\leq \max_{1\leq i\leq k}d_\infty(A_i,B_i)
\end{equation*}
for all $A_i,B_i\in\mathbb{P}$.
\end{proposition}
\begin{proof}
Let $t=\max_{1\leq i\leq k}d_\infty(A_i,B_i)$. Then $A_i\leq tB_i$ and $B_i\leq tA_i$ for all $1\leq i\leq k$, so by property 1 and 2
\begin{equation*}
\begin{split}
M(A_1,\ldots,A_k)\leq M(tB_1,\ldots,tB_k)=tM(B_1,\ldots,B_k)\\
M(B_1,\ldots,B_k)\leq M(tA_1,\ldots,tA_k)=tM(A_1,\ldots,A_k)
\end{split}
\end{equation*}
i.e., 
\begin{equation*}
\begin{split}
M(A_1,\ldots,A_k)\leq \max_{1\leq i\leq k}d_\infty(A_i,B_i)M(B_1,\ldots,B_k)\\
M(B_1,\ldots,B_k)\leq \max_{1\leq i\leq k}d_\infty(A_i,B_i)M(A_1,\ldots,A_k).
\end{split}
\end{equation*}

\end{proof}

\section{Generalized operator means via contraction principle}

In \cite{limpalfia} Lim and the author defined the one parameter family of matrix power means $P_s(\omega;{\Bbb A})$ as the unique positive definite solution of the equation
\begin{equation}\label{powermeanequ2}
X=\sum_{i=1}^{k}w_{i}X\#_sA_i
\end{equation}
where $s\in[-1,1], w_i>0, \sum_{i=1}^kw_i=1$ and $A_i\in\mathbb{P}$ and
\begin{equation*}
A\#_sB=A^{1/2}\left(A^{-1/2}BA^{-1/2}\right)^{s}A^{1/2}
\end{equation*}
is again the weighted geometric mean. Existence and uniqueness of the solution of \eqref{powermeanequ2} follow from the fact that the function
\begin{equation*}
f(X)=\sum_{i=1}^{k}w_{i}X\#_sA_i
\end{equation*}
is a strict contraction for $s\in[-1,1],s\neq 0$ with respect to Thompson's part metric \cite{limpalfia}.

Consider the following one-parameter family of real functions:
\begin{equation}\label{principalf}
f_{s,t}(x)=\frac{[(1-t)(1-s)+t]x+s(1-t)}{(1-t)(1-s)x+t+s(1-t)}
\end{equation}
for $t,s\in[0,1]$. By simple calculation one finds
\begin{equation}\label{eq:principalrepr}
\begin{split}
f_{s,t}(x)&=\frac{1}{t+s(1-t)}\left\{s(1-t)+\frac{t}{(1-t)(1-s)+[t+s(1-t)]x^{-1}}\right\}\\
&=\frac{1}{t+s(1-t)}\left[s(1-t)+th_{t+s(1-t)}(x)\right].
\end{split}
\end{equation}
Similarly we have for the inverse function $f_{s,t}^{-1}$ that
\begin{equation}\label{eq:principalreprinv}
f_{s,t}^{-1}(x)=\frac{[1-(1-t)s]x-(1-t)}{-(1-t)s(1-s)x+s+(1-s)t}.
\end{equation}

\begin{lemma}\label{L:principalf}
The real functions $f_{s,t}(x)$ for any fixed $t,s\in[0,1]$ defined by \eqref{principalf} are operator monotone and positive (as functions of $x$) for all $x\in(0,\infty)$.
\end{lemma}
\begin{proof}
Positivity of $f_{s,t}$ is clear, operator monotonicity follows from the fact that the functions
$$\frac{ax+b}{cx+d}$$
are operator monotone for real $x\neq -\frac{d}{c}$ if $ad-bc>0$, see for example \cite{bendatsherman}.

Alternatively one can directly arrive at the conclusion for $f_{s,t}$ by looking at \eqref{eq:principalrepr}, so that $f_{s,t}$ is a convex combination of two operator monotone functions.
\end{proof}

Due to Lemma~\ref{L:principalf} the functions $f_{s,t}\in\mathfrak{m}$ are representing functions of operator means $M_{s,t}(A,B)$ in $\mathfrak{M}$:
\begin{equation}\label{principalmean}
M_{s,t}(A,B):=A^{1/2}f_{s,t}\left(A^{-1/2}BA^{-1/2}\right)A^{1/2}.
\end{equation}

There is an important representation that holds for $f_{s,t}$.

\begin{proposition}\label{P:reprprincipalf}
For all $s,t\in[0,1]$ we have
\begin{equation*}
f_{s,t}(x)=\l_s^{-1}\left(t\l_s(x)\right),
\end{equation*}
where $\l_s(x):=\frac{x-1}{(1-s)x+s}$ and $\l_s^{-1}$ is its inverse function. Moreover $f_{s,t}\in\mathfrak{m}(t)$.
\end{proposition}
\begin{proof}
By simple computation.
\end{proof}

\begin{remark}
Notice that the functions $\l_s$ are the ones that occur in the integral representation given in Corollary~\ref{C:intLrepr} for the functions in $\mathfrak{L}$. This basic observation will be of fundamental importance for us.
\end{remark}

There are some other basic properties that follow from Proposition~\ref{P:reprprincipalf}. For example we have a semi-group property for the functions $f_{s,t}$:
\begin{equation}\label{semigrpproperty}
f_{s,t_1}\circ f_{s,t_2}(x)=f_{s,t_1t_2}(x).
\end{equation}

We will make use of the operator mean corresponding to the transpose of $f_{s,t}(x)$ i.e.,
\begin{equation*}
\begin{split}
xf_{s,t}(1/x)&=\frac{1}{t+s(1-t)}\left\{s(1-t)x+\frac{t}{(1-t)(1-s)x^{-1}+[t+s(1-t)]}\right\}\\
&=\frac{1}{t+s(1-t)}\left[s(1-t)x+th_{(1-t)(1-s)}(x)\right].
\end{split}
\end{equation*}

\begin{proposition}\label{P:principalcontr}
Let $f(X):=M_{s,t}(X,A)$ with $s\in[0,1]$, $t\in(0,1]$ arbitrary. Then the mapping $f(X)$ is a strict contraction on $\overline{B}_A(r)$ for all $r<\infty$ with contraction coefficient
\begin{equation*}
\rho=\rho_1+\frac{\log\frac{be^{-r}+ae^{2r(1-\rho_1)}}{be^{-r}+a}}{2r},
\end{equation*}
where $\rho_1=\frac{\log\frac{e^{3r}(1-t)+t}{e^{r}(1-t)+t}}{2r}$ and $a=\frac{s(1-t)}{t+s(1-t)}$, $b=\frac{t}{t+s(1-t)}$.
\end{proposition}
\begin{proof}
Consider the functions
\begin{eqnarray*}
f_1(X)&:=&\left\{(1-t)(1-s)X^{-1}+[t+s(1-t)]A^{-1}\right\}^{-1}\\
f_2(X)&:=&X.
\end{eqnarray*}
Then $M_{s,t}(X,A)=\frac{1}{t+s(1-t)}[tf_1(X)+s(1-t)f_2(X)]$ i.e., a convex combination. Now use Lemma~\ref{arithmeticcontraction} to conclude that $f_1$ is a strict contraction on $\overline{B}_A(r)$ with contraction coefficient
$$\rho_1=\frac{\log\frac{e^{3r}(1-t)(1-s)+t+s(1-t)}{e^{r}(1-t)(1-s)+t+s(1-t)}}{2r}\overset{(*)}{\leq}\frac{\log\frac{e^{3r}(1-t)+t}{e^{r}(1-t)+t}}{2r}.$$
The other function $f_2$ is nonexpansive. Now we have
\begin{eqnarray*}
d_\infty(f_1(X),f_2(X))&=&d_\infty\left(\left\{(1-t)(1-s)X^{-1}+[t+s(1-t)]A^{-1}\right\}^{-1},X\right)\\
&=&d_\infty\left((1-t)(1-s)I+[t+s(1-t)]X^{1/2}A^{-1}X^{1/2},I\right)\\
&\leq& d_\infty(X,A),
\end{eqnarray*}
where the last inequality follows from
\begin{equation*}
\begin{split}
&e^{d_\infty\left((1-t)(1-s)I+[t+s(1-t)]X^{1/2}A^{-1}X^{1/2},I\right)}\leq\\
&\max\left\{\|(1-t)(1-s)I+[t+s(1-t)]X^{1/2}A^{-1}X^{1/2}\|_\infty,\right.\\
&\left.\|(1-t)(1-s)I+[t+s(1-t)]X^{-1/2}AX^{-1/2}\|_\infty\right\}\\
&\leq(1-t)(1-s)+[t+s(1-t)]e^{d_\infty(X,A)}\\
&\leq e^{d_\infty(X,A)}\\
&\leq e^r.
\end{split}
\end{equation*}
This means that $d_\infty(f_1(X),f_2(X))\leq r$.

Let $X,Y\in\overline{B}_A(r)$. By Proposition~\ref{weightedthompson} we have that
\begin{equation*}
\begin{split}
&e^{d_\infty(M_{s,t}(X,A),M_{s,t}(Y,A))}\\
&\leq\max\left\{\frac{ae^{d_\infty(f_2(X),f_2(Y))}+e^{-d_\infty(f_1(X),f_2(X))}be^{d_\infty(f_1(X),f_1(Y))}}{a+e^{-d_\infty(f_1(X),f_2(X))}b}\right.,\\
&\left.\frac{ae^{d_\infty(f_2(X),f_2(Y))}+e^{-d_\infty(f_1(Y),f_2(Y))}be^{d_\infty(f_1(X),f_1(Y))}}{a+e^{-d_\infty(f_1(Y),f_2(Y))}b}\right\}.
\end{split}
\end{equation*}
Since $f_1$ is a strict contraction on $\overline{B}_A(r)$ with contraction coefficient $\rho_1$ and also $d_\infty(f_1(X),f_2(X))\leq r$ for all $X\in\overline{B}_A(r)$, hence we have
\begin{equation*}
e^{d_\infty(M_{s,t}(X,A),M_{s,t}(Y,A))}\overset{(*)}{\leq}\frac{ae^{d_\infty(X,Y)}+e^{-r}be^{\rho_1d_\infty(X,Y)}}{a+e^{-r}b}
\end{equation*}
using a similar argument as we did to obtain \eqref{eq:usingstar}.
Now we seek $0<\rho<1$ such that
\begin{equation*}
\begin{split}
d_\infty(M_{s,t}(X,A),M_{s,t}(Y,A)))&\leq\log\left(\frac{ae^{d_\infty(X,Y)}+e^{-r}be^{\rho_1d_\infty(X,Y)}}{a+e^{-r}b}\right)\\
&\leq \rho d_\infty(X,Y)
\end{split}
\end{equation*}
for all $X,Y\in\overline{B}_I(r)$ i.e., $d_\infty(X,Y)\leq 2r$. It therefore suffices to find the maximum of the function
\begin{equation*}
f(x)=\frac{\log\left(\frac{e^{\rho_1x}be^{-r}+ae^{x}}{be^{-r}+a}\right)}{x}
\end{equation*}
on the interval $[0,2r]$. First routine calculations show that 
\begin{equation*}
\lim_{x\to 0+}\frac{\log\left(\frac{e^{\rho_1x}be^{-r}+ae^{x}}{be^{-r}+a}\right)}{x}=\frac{\rho_1be^{-r}+a}{be^{-r}+a}<1.
\end{equation*}
Then the maximization problem is the same as finding the smallest $\rho<1$ such that
\begin{equation*}
\log\left(\frac{x^{\rho_1}be^{-r}+ax}{be^{-r}+a}\right)\leq\rho\log(x)
\end{equation*}
on the transformed interval $[1,e^{2r}]$. But that is equivalent to
\begin{equation*}
\frac{x^{\rho_1}be^{-r}}{be^{-r}+a}+\frac{ax}{be^{-r}+a}\leq x^{\rho},
\end{equation*}
which by substitution with $y:=x^{1-\rho_1}$ is equivalent to requiring
\begin{equation*}
\frac{be^{-r}}{be^{-r}+a}+\frac{ay}{be^{-r}+a}\leq y^{\frac{\rho-\rho_1}{1-\rho_1}}
\end{equation*}
for $y\in[1,e^{2r(1-\rho_1)}]$.
Now similar considerations as in the end of the proof of Lemma~\ref{arithmeticcontraction} lead to that the smallest such $\frac{\rho-\rho_1}{1-\rho_1}$ is determined at the $e^{2r(1-\rho_1)}$ endpoint of the interval, hence
\begin{equation*}
\frac{\rho-\rho_1}{1-\rho_1}=\frac{\log\frac{be^{-r}+ae^{2r(1-\rho_1)}}{be^{-r}+a}}{2r(1-\rho_1)}.
\end{equation*}
In other words
\begin{equation*}
\rho=\rho_1+\frac{\log\frac{be^{-r}+ae^{2r(1-\rho_1)}}{be^{-r}+a}}{2r}
\end{equation*}
which is strictly less than 1.
\end{proof}

Now we would like to consider convex combinations of $M_{s,t}(A,B)$ which amounts to integrating with respect to a measure. We have already introduced the weak operator Pettis integral in Definition~\ref{D:pettis} which was sufficient for our purpose so far. We will consider probability measures, but now we change the point of view and instead we mostly consider the measures directly given in $\mathbb{P}$, in other words the push forward measures under the injective continuous map from the probability space to $\mathbb{P}$.

\begin{definition}
Let $\mathscr{P}(\mathbb{P})$ denote the set of all probability measures with bounded support in $\mathbb{P}$ on the $\sigma$-algebra generated by the open sets of $\mathbb{P}$ in the norm topology.
\end{definition}

\begin{corollary}
Let $\sigma\in\mathscr{P}(\mathbb{P})$. Then any norm/strong/weak continuous function $f:\supp\sigma\to\mathbb{P}$ with bounded range is weak operator Pettis integrable with respect to $\sigma$.
\end{corollary}
\begin{proof}
Let $x,y\in E$. Then $\omega\mapsto\left\langle f(\omega)x,y\right\rangle$ is a continuous real valued function, hence measurable. The range of $f$ is bounded, so there exists a $K<\infty$ such that
$$f(\omega)\leq KI$$
for all $\omega\in\supp\sigma$. Hence by the Cauchy-Schwarz inequality we have
$$|\left\langle f(\omega)x,y\right\rangle|\leq \|f(\omega)x\| \|y\|\leq K\|x\| \|y\|$$
for fixed $x,y\in E$ i.e., the continuous real function $\left\langle f(\omega)x,y\right\rangle$ is bounded, hence integrable by the dominated convergence theorem or Theorem 11.8 in \cite{aliprantis}. Therefore we can apply Lemma~\ref{L:weakoperatorPettis}.
\end{proof}

The above corollary ensures us that in the remaining parts of the paper all weak operator Pettis integrals exist, since all of our functions will be continuous and bounded. In particular we will consider various integrals of $\l_s(X)$ and $M_{s,t}(X,A)$ which have lower and upper bounds due to \eqref{eq:maxmininp}, \eqref{eq:logbounds} and \eqref{eq:logbounds2}, moreover they are norm continuous functions in each of their arguments.


Let $\mathscr{P}([0,1])$ denote the set of all probability measures over the interval $[0,1]$. $\mathscr{P}([0,1])$ is a subset of the Banach space of finite signed measures over the interval $[0,1]$. Also $\mathscr{P}(\mathbb{P})$ is a subset of the Banach space of finite signed measures over $\mathbb{P}$ and in both cases the norm is provided by the total variation.
\begin{notation}
$\mathscr{P}([0,1]\times\mathbb{P})$ denotes the set of all probability measures on $[0,1]\times\mathbb{P}$ with bounded support. For $\mu\in\mathscr{P}([0,1]\times\mathbb{P})$ we say that $\mu$ is supported on $\{s\}\times\{A\}$ for some $s\in[0,1]$ and $A\in\mathbb{P}$, if $\{s\}\times\{A\}\subseteq\supp\mu$. We also say that $\mu$ is supported in $[0,1]\times\{A\}$, if there exists some non-empty subset $I\subseteq[0,1]$ such that $I\times\{A\}\subseteq\supp\mu$.
\end{notation}

\begin{definition}[Partial order for measures]\label{D:measureOrder}
Given a probability space $(\Omega,\mathcal{A},\pi)$, let $g_1:\Omega\mapsto\mathbb{P}$ and $g_2:\Omega\mapsto\mathbb{P}$ be two $\pi$-measurable maps. Denote by $\alpha_1:=g_{1*}\pi$ and $\alpha_2:=g_{2*}\pi$ the pushforward measures. Then we denote by
$$\alpha_1\leq \alpha_2$$
if and only if $g_1(\omega)\leq g_2(\omega)$ for all $\omega\in\Omega$.

Similarly let $([0,1]\times\Omega,\mathcal{A},\mu)$ be a probability space. Let $f_1:[0,1]\times\Omega\mapsto [0,1]\times\mathbb{P}$ and $f_2:[0,1]\times\Omega\mapsto [0,1]\times\mathbb{P}$ be two $\mu$-measurable maps and let $\nu_1:=f_{1*}\mu$ and $\nu_2:=f_{2*}\mu$ denote the pushforward measures on $[0,1]\times\mathbb{P}$.
Then we denote by
$$\nu_1\leq \nu_2$$
if and only if for all fixed $s\in[0,1]$ and $\omega\in\Omega$ we have $[f_1(s,\omega)]_1=[f_2(s,\omega)]_1$ and $[f_1(s,\omega)]_2\leq[f_2(s,\omega)]_2$ with respect to the positive definite order.
\end{definition}

\begin{remark}
The above Definition~\ref{D:measureOrder} of ordering of measures in $\mathscr{P}([0,1]\times\mathbb{P})$ is a direct generalization of the ordering of $k$-tuples $\mathbb{A},\mathbb{B}\in\mathbb{P}^n$. In previous works \cite{lawsonlim1,limpalfia} $\mathbb{A}\leq\mathbb{B}$ was understood element-wise i.e., $A_i\leq B_i$ for all $1\leq i\leq k$, and only tuples with identical associated element-wise weights were compared. Under this joint order operator monotonicity of (weighted) multivariable means were derived. Let us consider a typical example. Let $\xi$ be a probability measure on $[0,1]$, let $(\Omega,\mathcal{A},\mu)$ be a probability space and let $f_1:\Omega\mapsto\mathbb{P}$ and $f_2:\Omega\mapsto\mathbb{P}$ be two $\mu$-measurable maps with $f_1(\omega)\leq f_2(\omega)$ for all $\omega\in\Omega$. Then denoting by $\nu_1:=f_{1*}\mu$ and $\nu_2:=f_{2*}\mu$ the pushforward measures, we have for the product measures $\xi\times\nu_1\leq \xi\times\nu_2$ in the sense of our order.
\end{remark}
\begin{remark}
From the cone-theoretic point of view, the order $\mu_1\leq\mu_2$ in Definition~\ref{D:measureOrder} should be defined by requiring $\int_{[0,1]\times\mathbb{P}}f(s,A)d\mu_1(s,A)\leq \int_{[0,1]\times\mathbb{P}}f(s,A)d\mu_2(s,A)$ for all measurable and integrable functions $f:[0,1]\times\mathbb{P}\to\mathbb{P}$. If we had adopted this order, then it is not hard to see that all of our operator means defined in the remaining parts of the paper would have been monotone with respect to this cone-theoretic order as well, which appears to be weaker then our adopted order above i.e., one can compare more pairs of measures using this cone-theoretic partial order. Also the proof of monotonicity would be along the same lines here below.
\end{remark}

The set $\mathscr{P}([0,1]\times\mathbb{P})$ is a subset of a Banach space of all finite signed measures on $[0,1]\times\mathbb{P}$ with the total variation norm. Now we need a Fubini type of result.

\begin{lemma}\label{L:fubini}
Let $\nu\in\mathscr{P}([0,1])$ and $\sigma\in\mathscr{P}(\mathbb{P})$. Then
\begin{eqnarray*}
\int_{\mathbb{P}}\int_{[0,1]}\l_s\left(X^{-1/2}AX^{-1/2}\right)d\nu(s)d\sigma(A)\\
=\int_{[0,1]\times\mathbb{P}}\l_s\left(X^{-1/2}AX^{-1/2}\right)d(\nu\times\sigma)(s,A)\\
=\int_{[0,1]}\int_{\mathbb{P}}\l_s\left(X^{-1/2}AX^{-1/2}\right)d\sigma(A)d\nu(s).
\end{eqnarray*}
\end{lemma}
\begin{proof}
First note that the function $\l_s\left(X^{-1/2}AX^{-1/2}\right)$ is jointly norm continuous in $\mathbb{P}\times[0,1]$, so therefore the real function $\left\langle \l_s\left(X^{-1/2}AX^{-1/2}\right)u,v\right\rangle$ is also continuous on $[0,1]\times\mathbb{P}$ for any $u,v\in E$, hence it is measurable with respect to $\nu\times\sigma$, also it is bounded since the support of $\nu\times\sigma$ is bounded i.e., $\l_s\left(X^{-1/2}AX^{-1/2}\right)$ is weak operator Pettis integrable. The rest of the assertion follows from Fubini's theorem for the Lebesgue integrable function $\left\langle \l_s\left(X^{-1/2}AX^{-1/2}\right)u,v\right\rangle$.
\end{proof}

In what follows after the above preparations we will study an analogue of \eqref{powermeanequ2}.

\begin{lemma}\label{powercontract}
Let $\mu\in\mathscr{P}([0,1]\times\mathbb{P})$ and $t\in(0,1]$. Then the function
\begin{equation}\label{meanequf}
f(X)=\int_{[0,1]\times\mathbb{P}}M_{s,t}(X,A)d\mu(s,A)
\end{equation}
is a strict contraction with respect to the Thompson metric $d_\infty(\cdot,\cdot)$ on every bounded $S\subseteq\mathbb{P}$ such that $\supp\mu\subseteq [0,1]\times S$.
\end{lemma}
\begin{proof}
Fix $t\in(0,1]$. Choose a large enough $r<\infty$ such that for each $A\in\mathbb{P}$ for which $\mu$ is supported in $[0,1]\times\{A\}$, we have $S\subseteq\overline{B}_A(r)$. We can do that since $\supp\mu$ is bounded in $[0,1]\times\mathbb{P}$ by definition, also therefore we have uniformly 
$$e^{k}I\leq A\leq e^KI$$
for all $A\in\mathbb{P}$ such that $\mu$ is supported in $[0,1]\times\{A\}$ and for some constants $-\infty<k\leq K<\infty$. By Proposition~\ref{P:principalcontr} $g_{s,t}(X)=M_{s,t}(X,A)$ is a strict contraction on $\overline{B}_A(r)$, with contraction coefficient
\begin{equation*}
\rho=\rho_1+\frac{\log\frac{be^{-r}+ae^{2r(1-\rho_1)}}{be^{-r}+a}}{2r},
\end{equation*}
where $\rho_1=\frac{\log\frac{e^{3r}(1-t)+t}{e^{r}(1-t)+t}}{2r}$ and $a=\frac{s(1-t)}{t+s(1-t)}$, $b=\frac{t}{t+s(1-t)}$. Since $t\leq b\leq 1$ and $0\leq a\leq 1-t$ we have by using (*) that
\begin{equation*}
\rho\leq\rho_1+\frac{\log\frac{te^{-r}+(1-t)e^{2r(1-\rho_1)}}{te^{-r}+(1-t)}}{2r}
\end{equation*}
which is strictly less then $1$. Hence using property (4') in Lemma~\ref{L:infinite4} we have that
\begin{equation*}
d_\infty(f(X),f(Y))\leq \rho d_\infty(X,Y)
\end{equation*}
for all $X,Y\in S$ i.e., $f(X)$ is a strict contraction on $S$.
\end{proof}

\begin{proposition}\label{P:uniquesol}
Let $\mu\in\mathscr{P}([0,1]\times\mathbb{P})$, $t\in(0,1]$. Then the equation
\begin{equation}\label{meanequ21}
X=\int_{[0,1]\times\mathbb{P}}M_{s,t}(X,A)d\mu(s,A)
\end{equation}
has a unique positive definite solution in $\mathbb{P}$.
\end{proposition}
\begin{proof}
By Lemma~\ref{powercontract}, for every large enough bounded subset $S\subseteq\mathbb{P}$ the function $f(X)$ given in \eqref{meanequf} is a strict contraction on $S$. Suppose now that $S:=\overline{B}_I(r)$ for any $r<\infty$ such that $\supp\mu\subseteq [0,1]\times S$.
Then we claim that $f(S)\subseteq S$. We have the following bounds given by \eqref{eq:maxmininp}:
$$\left[(1-t)X^{-1}+tA^{-1}\right]^{-1}\leq M_{s,t}(X,A)\leq (1-t)X+tA.$$
Since $S=\overline{B}_I(r)$ it follows that for any $X\in S$ we have $e^{-r}I\leq X\leq e^{r}I$. Since $\supp\mu\subseteq [0,1]\times S$ we also have $e^{-r}I\leq A\leq e^{r}I$ for all $A\in\mathbb{P}$ such that $\mu$ is supported in $[0,1]\times\{A\}$. Therefore by property (ii) in Definition~\ref{symmean} we have
\begin{eqnarray*}
e^{-r}I&=&\left[(1-t)(e^{-r}I)^{-1}+t(e^{-r}I)^{-1}\right]^{-1}\\
&\leq &\left[(1-t)X^{-1}+t(e^{-r}I)^{-1}\right]^{-1}\\
&\leq &\int_{[0,1]\times\mathbb{P}}M_{s,t}(X,A)d\mu(s,A)
\end{eqnarray*}
for all $X\in S$ and $A\in\mathbb{P}$ such that $\mu$ is supported in $[0,1]\times\{A\}$. A similar argument using the weighted arithmetic mean leads to
\begin{eqnarray*}
\int_{[0,1]\times\mathbb{P}}M_{s,t}(X,A)d\mu(s,A)\leq e^{r}I.
\end{eqnarray*}
This yields that $e^{-r}I\leq f(X)\leq e^{r}I$ for all $X\in S=\overline{B}_I(r)$, hence $f(S)\subseteq S$.

Now the iterates of $f^{\circ n}(X)$ stay in $S$ if $X\in S$ and $S$ is closed. Also $f$ is a strict contraction on $S$, therefore by Banach's fixed point theorem $f(X)$ has a unique fixed point in $S$, so equation \eqref{meanequ21} has a unique positive definite solution in $S$. Since $S=\overline{B}_I(r)$ was an arbitrary, large enough bounded subset of $\mathbb{P}$ such that $\supp\mu\subseteq [0,1]\times S$, it follows that the same holds on all of $\mathbb{P}$.
\end{proof}

\begin{definition}[Induced Operator Mean]
Let $\sigma\in\mathscr{P}(\mathbb{P})$, $t\in(0,1]$ and $\nu$ be a probability measure on $[0,1]$. We denote by $L_{t,\nu}(\sigma)$ the unique solution $X\in\mathbb{P}$ of the equation
\begin{equation}\label{meanequ3}
X=\int_{\mathbb{P}}\int_{[0,1]}M_{s,t}(X,A)d\nu(s)d\sigma(A).
\end{equation}
We call $L_{t,\nu}(\sigma)$ the $\sigma$-weighted $\nu$-induced operator mean.

Let $\mu\in\mathscr{P}([0,1]\times\mathbb{P})$ and $t\in(0,1]$. Then we denote by $L_{t}(\mu)$ the unique solution $X\in\mathbb{P}$ of the equation
\begin{equation}\label{meanequ31}
X=\int_{[0,1]\times\mathbb{P}}M_{s,t}(X,A)d\mu(s,A).
\end{equation}
We call $L_{t}(\mu)$ the $\mu$-weighted induced operator mean.
\end{definition}

In the above definition in the notations there should be no confusion, since the number of measures as arguments of $L$ should determine which operator mean is intended.

\begin{remark}\label{monotone}
Let $f(X)$ be defined by \eqref{meanequf}. Then by the monotonicity of $M_{s,t}$, $f$ is monotone: $X\leq Y$ implies that $f(X)\leq f(Y)$.
\end{remark}



For any $X\in \mathrm{GL}(E)$ and $\mu\in\mathscr{P}([0,1]\times\mathbb{P})$ we will use the notation
$$(X\mu X^{*})(s,A):=\mu\left(s,X^{-1}A(X^{*})^{-1}\right),$$
similarly for $\sigma\in\mathscr{P}(\mathbb{P})$
$$(X\sigma X^{*})(A):=\sigma\left(X^{-1}A(X^{*})^{-1}\right).$$
Integrating with respect to the above two measures is not a problem due to a generalized form of the change of variables formula that holds for the Lebesgue integral and hence trivially for the weak operator Pettis integral, see for example Theorem 2.26 in \cite{bashirov}. Also weak operator Pettis integration with respect to other transformed measures that we will see later is also permitted due to the continuity of the mappings.

\begin{proposition}\label{MPro}
Let $\mu,\mu_1,\mu_2\in\mathscr{P}([0,1]\times\mathbb{P})$ and $t\in(0,1]$. Then
\begin{itemize}
 \item[(1)] $L_{t}(\mu)=A$ if $\supp\mu\subseteq [0,1]\times\{A\}$ for an $A\in\mathbb{P}$;
 \item[(2)] $L_{t}(\mu_1)\leq L_{t}(\mu_2)$ if $\mu_1\leq
 \mu_2;$
 \item[(3)] $L_{t}(X\mu X^{*})=XL_{t}(\mu)X^{*}$ for any $X\in \mathrm{GL}(E);$
 \item[(4)] Suppose $\int_{[0,1]\times\mathbb{P}}M_{s,t}(X,A)d\mu_1(s,A)\leq \int_{[0,1]\times\mathbb{P}}M_{s,t}(X,A)d\mu_2(s,A)$ for $\mu_1,\mu_2\in\mathscr{P}([0,1]\times\mathbb{P})$. Then $L_{t}(\mu_1)\leq L_{t}(\mu_2);$
 \item[(5)] If $0<t_1\leq t_2\leq 1$ then $L_{t_1}(\mu)\leq L_{t_2}(\mu);$
 \item[(6)] If $d\mu_2(s,A)=d\mu_1(s,g(s,A))$ where $g$ is measurable for fixed $s$, then $(1-u)L_{t}(\mu_1)+uL_{t}(\mu_2)\leq
 L_{t}((1-u)\mu_1+u\mu_2)$ for any $u\in [0,1];$
 \item[(7)] If $d\mu_2(s,A)=d\mu_1(s,g(s,A))$ where $g$ is measurable for all fixed $s$, then $d_{\infty}(L_{t}(\mu_1), L_{t}(\mu_2))\leq \underset{\mu_2\text{ is supported on }\{s\}\times\{A\}}{\sup}\{d_{\infty}(A,g(s,A))\};$
 \item[(8)]$\Phi(L_{t}(\mu))\leq
L_{t}(\Phi(\mu))$ for any measurable positive unital linear map $\Phi$, where $\Phi(\mu)(s,A):=\mu(s,\Phi^{-1}(A)).$
\end{itemize}
\end{proposition}
\begin{proof}
(1) By \eqref{meanequ31} we have $X=\int_{[0,1]\times\mathbb{P}}M_{s,t}(X,A)d\mu(s,A)$ and using that $M_{s,t}(A,A)=A$ we see that $X=A$ is a, and by uniqueness, the solution of \eqref{meanequ31}.

\vspace{2mm}

(2) Define $$ f(X):=\int_{[0,1]\times\mathbb{P}}M_{s,t}(X,A)d\mu_1(s,A)$$ and $$
g(X):=\int_{[0,1]\times\mathbb{P}}M_{s,t}(X,A)d\mu_2(s,A). $$ Then $L_{t}(\mu_1)=\lim_{l\to \infty}f^{\circ l}(X)$ and $L_{t}(\mu_2)=\lim_{l\to \infty}g^{\circ l}(X)$ for any $X\in {\Bbb P}$, by the
Banach fixed point theorem. By the monotonicity of $M_{s,t}\in\mathfrak{M}$ we have
$M_{s,t}(X,g(s,A))\leq M_{s,t}(X,A)$ and if we integrate this we get
$f(X)\leq g(X)$ for all $X\in {\Bbb P},$ similarly follows that $f(X)\leq f(Y),
g(X)\leq g(Y)$ whenever $X\leq Y.$ Let $X_{0}:=L_{t}(\mu_1)$. Then $f(X_{0})\leq
g(X_{0})$ and $f^{\circ 2}(X_{0})=f(f(X_{0}))\leq g(f(X_{0}))\leq
g^{\circ 2}(X_{0}).$ Inductively, we have $f^{\circ l}(X_{0})\leq g^{\circ l}(X_{0})$
for all $l\in {\Bbb N}.$ Therefore, $L_{t}(\mu_1)=\lim_{l\to\infty}f^{\circ l}(X_{0})\leq \lim_{l\to\infty}
g^{\circ l}(X_{0})=L_{t}(\mu_2).$

 \vspace{2mm}

 (3) We have $XM_{s,t}(A,B)X^{*}=M_{s,t}(XAX^{*},XBX^{*})$ and applying this to the defining equation \eqref{meanequ3} the property follows.


 (4) Define $$f(X):=\int_{[0,1]\times\mathbb{P}}M_{s,t}(X,A)d\mu_1(s,A)$$ and $$
g(X):=\int_{[0,1]\times\mathbb{P}}M_{s,t}(X,A)d\mu_2(s,A). $$ Then $L_{t}(\mu_1)=\lim_{l\to \infty}f^{\circ l}(X)$ and $L_{t}(\mu_2)=\lim_{l\to \infty}g^{\circ l}(X)$ for any $X\in {\Bbb P}$, by the
Banach fixed point theorem. By assumption we also have that $f(X)\leq g(X)$. Let $Y:=L_{t}(\mu_1)$. Then we have $Y=f(Y)\leq g(Y)$ and inductively $Y=f^{\circ l}(Y)\leq g^{\circ l}(Y)$, hence $L_{t}(\mu_1)=Y\leq \lim_{l\to\infty}g^{\circ l}(Y)=L_{t}(\mu_2)$.

\vspace{2mm}

 (5) By Proposition~\ref{P:reprprincipalf} for all $s,t\in[0,1]$ we have
 $$f_{s,t}(x)=\l_s^{-1}\left(t\l_s(x)\right),$$
 so for $0<t_1<t_2\leq 1$ we have $f_{s,t_1}(x)\leq f_{s,t_2}(x)$, since $\l_s^{-1}$ is a monotone convex increasing function, which follows from the operator monotonicity, hence concavity of $\l_s$. Then by Proposition~\ref{P:ordermean} we get
$$\int_{[0,1]\times\mathbb{P}}M_{s,t_1}(A,B)d\mu(s,A)\leq \int_{[0,1]\times\mathbb{P}}M_{s,t_2}(A,B)d\mu(s,A).$$
Now we may apply a similar argument as in the proof of property (4) to conclude that $L_{t_1}(\mu)\leq L_{t_2}(\mu).$

\vspace{2mm}

 (6) Let $X=L_{t}(\mu_1)$ and $Y=L_{t}(\mu_2).$ For $u\in [0,1],$ we set $Z_{u}=(1-u)X+uY.$ Let $$f(Z)=\int_{[0,1]\times\mathbb{P}}M_{s,t}(Z,(1-u)A+ug(s,A))d\mu_1(s,A).$$ Then by the joint concavity of two-variable operator means (Theorem 3.5 \cite{kubo})
\begin{eqnarray*}
Z_{u}&=&(1-u)X+uY\\
&=&\int_{[0,1]\times\mathbb{P}}\left[(1-u)M_{s,t}(X,A)+uM_{s,t}(Y,g(s,A))\right]d\mu_1(s,A)\\
&\leq&\int_{[0,1]\times\mathbb{P}}M_{s,t}((1-u)X+uY,(1-u)A+ug(s,A))d\mu_1(s,A)\\
&=&f(Z_{u}).
\end{eqnarray*}
Inductively, $Z_{u}\leq f^{\circ l}(Z_{u})$ for all $l\in {\Bbb N}.$
Therefore, $ (1-u)L_{t}(\mu_1)+uL_{t}(\mu_2)=Z_{u}\leq L_{t}((1-u)\mu_1+u\mu_2).$

\vspace{2mm}

 (7) Follows from a similar argument to the proof of Proposition~\ref{nonexpansivemean} using property (2) and (3).

\vspace{2mm}

 (8) Note that $\Phi(M_{s,t}(A,B))\leq
M_{s,t}(\Phi(A),\Phi(B))$ for any $A,B>0$ by Proposition~\ref{positivemapthm}. Then
\begin{equation}\label{E:L1}
\begin{split}
\Phi(L_{t}(\mu))&=\int_{[0,1]\times\mathbb{P}}\Phi(M_{s,t}(L_{t}(\mu),A))d\mu(s,A)\\
&\leq \int_{[0,1]\times\mathbb{P}}M_{s,t}(\Phi(L_{t,\nu_1}(\sigma_1),A)),\Phi(A))d\mu(s,A).
\end{split}
\end{equation}
Define $$f(X)=\int_{[0,1]\times\mathbb{P}}M_{s,t}(X,\Phi(A))d\mu(s,A).$$ Then
$\lim_{l\to\infty}f^{\circ l}(X)=L_{t}(\Phi(\mu))$ for any
$X>0.$ By \eqref{E:L1}, $f(\Phi(L_{t}(\mu)))\geq \Phi(L_{t}(\mu)).$ Since $f$
is monotonic, $f^{\circ l}(\Phi(L_{t}(\mu)))\geq \Phi(L_{t}(\mu))$ for all $l\in
{\Bbb N}.$ Thus
\begin{eqnarray*}
L_{t}(\Phi(\mu))=\lim_{l\to\infty}f^{\circ l}(\Phi(L_{t}(\mu)))\geq
\Phi(L_{t}(\mu)).
\end{eqnarray*}
\end{proof}

\begin{corollary}\label{cor:opmean}
Suppose $\nu\in\mathscr{P}([0,1])$ and $\sigma\in\mathscr{P}(\mathbb{P})$ and assume that $\sigma(X):=(1-w)\delta_{A}(X)+w\delta_{B}(X)$ with $w\in(0,1)$ where $\delta_{A}(X)$ denotes the Dirac delta supported on ${A}$. 
Then $L_{t}(\nu\times\sigma)$ is an operator mean in the two variables $(A,B)$ i.e., $L_{t}(\nu\times\sigma)\in\mathfrak{M}$.
\end{corollary}
\begin{proof}
By property (3) in Proposition~\ref{MPro} it follows that
$$L_{t}(\nu\times\sigma)=A^{1/2}g(A^{-1/2}BA^{-1/2})A^{1/2}$$
and property (1) yields that $g(I)=I$. By Lemma~\ref{powercontract} we have that
\begin{equation*}
\lim_{l\to \infty}f^{\circ l}(X)=g(A^{-1/2}BA^{-1/2})
\end{equation*}
for all $X\in\mathbb{P}$ where
$$f(X)=\int_{[0,1]}(1-w)M_{s,t}(X,I)+wM_{s,t}(X,A^{-1/2}BA^{-1/2})d\nu(s).$$
We can choose $X=I$ and then
\begin{equation*}
\lim_{l\to \infty}f^{\circ l}(I)=g(C)
\end{equation*}
where $C=A^{-1/2}BA^{-1/2}$. Also by simple calculation we have that
$$f(X)=(1-w)Xf_{s,t}(X^{-1})+wXf_{s,t}(X^{-1}C).$$
By property (2) in Proposition~\ref{MPro}, $g$ is operator monotone.  Moreover $f^{\circ l}(I)$ is an analytic real map in the single variable $C$ for all $l$, moreover the net $f^{\circ l}(I)$ converges uniformly on bounded subsets of $\mathbb{P}$ due to the strict contraction property of $f(X)$. Hence the pointwise limit $\lim_{l\to\infty}f^{\circ l}(1)$ for positive real (scalar) $C$ is a continuous real map as well and is identical to $g$ by the properties of the functional calculus of self-adjoint operators, since the net $f^{\circ l}(I)$ converges in norm for all $C$ (the topology generated by the metric $d_\infty$ agrees with the relative Banach space topology \cite{thompson}). It is also easy to see that $g$ is positive on $(0,\infty)$ and $g(1)=1$, hence $g$ is an operator monotone function in $\mathfrak{m}$. So by Theorem 3.2 in \cite{kubo} we get that $L_{t}(\nu\times\sigma)=A^{1/2}g(A^{-1/2}BA^{-1/2})A^{1/2}$ is an operator mean in the sense of Definition~\ref{symmean}.
\end{proof}

\section{Generalized Karcher equations and one parameter families of operator means}
In this section we generalize the results of \cite{limpalfia,lawsonlim1} which were given for the one parameter family of matrix power means. We will provide solutions of nonlinear operator equations that are given in Definition~\ref{D:genKarcherequ}, this time considered in the setting of the full (possibly infinite dimensional) cone $\mathbb{P}$. Let us repeat the definition once more:
\begin{definition}[Generalized Karcher equation]\label{generalizedkarcherequ}
Let $\mu\in\mathscr{P}([0,1]\times\mathbb{P}).$ The generalized Karcher equation is the operator equation
\begin{equation*}
\int_{[0,1]\times\mathbb{P}}X^{1/2}\l_s(X^{-1/2}AX^{-1/2})X^{1/2}d\mu(s,A)=0
\end{equation*}
for $X\in\mathbb{P}$.
\end{definition}

\begin{proposition}\label{oneparameterfam}
The one parameter family of $\mu$-weighted operator means $L_{t}(\mu)$ are continuous for $t\in(0,1]$ on any bounded set $S\subseteq\mathbb{P}$ with respect to the topology generated by $d_\infty$ (the norm topology).
\end{proposition}
\begin{proof}
The induced operator means $L_{t}(\mu)$ are fixed points of mappings $f(X)$ given in \eqref{meanequf} which are strict contractions on any bounded subset of $S\subseteq\mathbb{P}$ according to Lemma~\ref{powercontract}. Therefore on every bounded set $S\subseteq\mathbb{P}$, $L_{t}(\mu)$ varies continuously in $t$ with respect to the topology generated by the metric $d_\infty$ due to the continuity of fixed points of pointwisely continuous families of strict contractions \cite{neeb}.
\end{proof}

\begin{lemma}\label{maxmininduced}
Let $\mu\in\mathscr{P}([0,1]\times\mathbb{P})$. Then for $t\in(0,1]$ we have
\begin{equation*}
kI\leq L_{t}(\mu)\leq \int_{[0,1]\times\mathbb{P}}Ad\mu(s,A),
\end{equation*}
where $k>0$ is such that $kI\leq B$ for any $B\in\mathbb{P}$ such that $\mu$ is supported in $[0,1]\times\{B\}$.
\end{lemma}
\begin{proof}
By Lemma~\ref{maxmininp} and Proposition~\ref{P:reprprincipalf} we have that
\begin{equation}\label{eq:maxmininducedd}
M_{s,t}(X,A)\leq (1-t)X+tA.
\end{equation}
which yields
\begin{equation*}
f(X):=\int_{[0,1]\times\mathbb{P}}M_{s,t}(X,A)d\nu(s,A)\leq (1-t)X+t\int_{[0,1]\times\mathbb{P}}Ad\mu(s,A).
\end{equation*}
If we define $g(X):=(1-t)X+t\int_{[0,1]\times\mathbb{P}}Ad\mu(s,A)$, then some simple calculation reveals that for any $X\in\mathbb{P}$ we have
$$\lim_{n\to\infty}g^{\circ n}(X)=\int_{[0,1]\times\mathbb{P}}Ad\mu(s,A).$$
From this we have
$$L_{t}(\mu)=\lim_{n\to\infty}f^{\circ n}(X)\leq \lim_{n\to\infty}g^{\circ n}(X)=\int_{[0,1]\times\mathbb{P}}Ad\mu(s,A).$$

The lower bound follows from the proof of Proposition~\ref{P:uniquesol} where for any large enough bounded ball $S\subseteq\mathbb{P}$ we have $f(S)\subseteq S$.
\end{proof}

Let us recall the strong topology on $\mathbb{P}$. The positive definite partial order $\leq$ is strongly continuous, so if $A_n\to A$, $B_n\to B$ and $A_n\leq B_n$ then $A\leq B$. Also if $A_n$ is a monotonically decreasing net in $\mathbb{P}$ with respect to $\leq$ and it is bounded from below, then it converges strongly to the infimum of $A_n$. Similarly if $B_n$ monotonically increases and is bounded from above, then $B_n$ converges strongly to its supremum \cite{weidman}.

\begin{theorem}\label{inducedconv}
Let $\mu\in\mathscr{P}([0,1]\times\mathbb{P})$. Then there exists $X_0\in\mathbb{P}$ such that
\begin{equation*}
\lim_{t\to 0+}L_{t}(\mu)=X_0.
\end{equation*}
\end{theorem}
\begin{proof}
By property (5) in Proposition~\ref{MPro} we have that $t\mapsto L_{t}(\mu)$ is a decreasing net bounded from below by $kI$ by Lemma~\ref{maxmininduced}, hence it is convergent in the strong topology as $t\to 0+$.
\end{proof}

\begin{definition}[Lambda operator means]
Let $\Lambda_\nu(\sigma):=\lim_{t\to 0+}L_{t,\nu}(\sigma)$ and call it the $\sigma$-weighted $\nu$-lambda operator mean.

Also let $\mu\in\mathscr{P}([0,1]\times\mathbb{P})$. Then we denote $\Lambda(\mu):=\lim_{t\to 0+}L_{t}(\mu)$ and call it the $\mu$-weighted lambda operator mean.
\end{definition}

\begin{remark}
If we take the one parameter family of matrix power means $P_t(\omega;{\Bbb A})$, which are weighted by finitely supported probability measures $\omega$ on the points ${\Bbb A}$, then it is known that $\lim_{t\to 0}P_t(\omega;{\Bbb A})$ is the Karcher mean $\Lambda(\omega;\mathbb{A})$ see \cite{limpalfia} for the finite dimensional setting and \cite{lawsonlim1} for the general infinite dimensional case.
\end{remark}

\begin{theorem}\label{lambdaprop}
Let $\mu,\mu_1,\mu_2\in\mathscr{P}([0,1]\times\mathbb{P})$ and $t\in(0,1]$. Then
\begin{itemize}
 \item[(1)] $\Lambda(\mu)=A$ if $\mu$ is only supported in $[0,1]\times\{A\}$;
 \item[(2)] $\Lambda(\mu_1)\leq \Lambda(\mu_2)$ if $\mu_1\leq
 \mu_2;$
 \item[(3)] $\Lambda(X\mu X^{*})=X\Lambda(\mu)X^{*}$ for any $X\in \mathrm{GL}(E);$
 \item[(4)] Suppose $$\int_{[0,1]\times\mathbb{P}}M_{s,t}(X,A)d\mu_1(s,A)\leq \int_{[0,1]\times\mathbb{P}}M_{s,t}(X,A)d\mu_2(s,A)$$
 for all $t\in[0,1]$. Then $\Lambda(\mu_1)\leq \Lambda(\mu_2);$
 \item[(5)] If $d\mu_2(s,A)=d\mu_1(s,g(s,A))$ where $g$ is measurable for fixed $s$, then $(1-u)\Lambda(\mu_1)+u\Lambda(\mu_2)\leq
 \Lambda((1-u)\mu_1+u\mu_2)$ for any $u\in [0,1];$
 \item[(6)] If $d\mu_2(s,A)=d\mu_1(s,g(s,A))$ where $g$ is measurable for all fixed $s$, then $d_{\infty}(\Lambda(\mu_1), \Lambda(\mu_2))\leq \underset{\mu_2\text{ is supported on }\{s\}\times\{A\}}{\sup}\{d_{\infty}(A,g(s,A))\};$
 \item[(7)]$\Phi(\Lambda(\mu))\leq
\Lambda(\Phi(\mu))$ for any measurable positive unital linear map $\Phi$, where $\Phi(\mu)(s,A):=\mu(s,\Phi^{-1}(A)).$
 \item[(8)]
$kI\leq\Lambda(\mu)\leq \int_{[0,1]\times\mathbb{P}}Ad\mu(s,A)$ where $k>0$ is such that $kI\leq B$ for any $B\in\mathbb{P}$ such that $\mu$ is supported in $[0,1]\times\{B\}$.
\end{itemize}
\end{theorem}
\begin{proof}
Each of the properties easily follows from Proposition~\ref{MPro} and Lemma~\ref{maxmininduced} by taking the limit $t\to 0+$.
\end{proof}

Now we turn to the study of the generalized Karcher equation
\begin{equation}\label{karcherequx}
\int_{[0,1]\times\mathbb{P}}\l_s(X^{-1/2}AX^{-1/2})d\mu(s,A)=0
\end{equation}
for a $\mu\in\mathscr{P}([0,1]\times\mathbb{P})$. We denote by $K(\mu)$ the set of all solutions $X$ of \eqref{karcherequx} in $\mathbb{P}$.

\begin{lemma}
Operator multiplication is strongly continuous on any bounded set.
\end{lemma}
\begin{proof}
Let $A_l\to A,B_l\to B$ strongly, and $\left\|A_l\right\|,\left\|B_l\right\|\leq K$. Then
\begin{equation*}
\left\|(A_lB_l-AB)x\right\|\leq \left\|A_l(B_l-B)x\right\|+\left\|(A_l-A)Bx\right\|\leq K\left\|(B_l-B)x\right\|+\left\|(A_l-A)Bx\right\|,
\end{equation*}
so $\left\|(A_lB_l-AB)x\right\|\to 0$ as well.

\end{proof}

\begin{lemma}
Let $Q$ be an open or closed subset of $\mathbb{R}$ and let $f:Q\to \mathbb{R}$ be continuous and bounded. Then $f$ is strong operator continuous on the set $S(E)$ of self adjoint operators with spectrum in $Q$.
\end{lemma}
\begin{proof}
Special case of Theorem 3.6 in \cite{kadison}.
\end{proof}

The consequece of the above is the following
\begin{lemma}
The functions
\begin{enumerate}
	\item $x^{-1}$,
	\item $f_{s,t}(x)=\l^{-1}_s(t\l_s(x))$ for $0\leq t\leq 1$,
	\item the mean $M_{s,t}(A,B)$ for $0\leq t\leq 1$,
\end{enumerate}
are strongly continuous on the order intervals $[e^{-m}I,e^{m}I]$ for any $m>0$.
\end{lemma}

\begin{lemma}\label{lemstrongconv}
Let $V\in S(E)$. Then
\begin{equation}
\lim_{(t,U)\to(0,V)}\frac{\l^{-1}_s(tU)-I}{t}=V,
\end{equation}
in the strong operator topology.
\end{lemma}
\begin{proof}
The function $\l^{-1}_s$ is the inverse of $\l_s$ and simple calculation shows that
$$\l_s^{-1}(x)=1+\frac{x}{1-(1-s)x}.$$
Now simple calculation shows that
$$\lim_{(t,U)\to(0,V)}\frac{\l^{-1}_s(tU)-I}{t}=\lim_{(t,U)\to(0,V)}U[I-(1-s)tU]^{-1}=V.$$
\end{proof}

\begin{theorem}\label{karchersatisfied}
The lambda operator mean $\Lambda(\mu)$ satisfies the generalized Karcher equation
\begin{equation*}
\int_{[0,1]\times\mathbb{P}}X^{1/2}\l_s(X^{-1/2}AX^{-1/2})X^{1/2}d\mu(s,A)=0.
\end{equation*}
\end{theorem}
\begin{proof}
For $0<t\leq 1$ let $X_t=L_{t}(\mu)$ and $X_0:=\Lambda(\mu)=\lim_{t\to 0+}L_{t}(\mu)$. By Theorem~\ref{inducedconv} $X_t\to X_0$ strongly monotonically as $t\to 0+$ and $kI\leq X_0\leq X_t\leq \int_{[0,1]\times\mathbb{P}}Ad\mu(s,A)$.
Now choose $m>0$ such that $A,X_{t},X_0\in[e^{-m}I,e^{m}I]$ for all $A$ such that $\mu$ is supported in $[0,1]\times\{A\}$. Then also $X_t\in[e^{-m}I,e^{m}I]$ for $0\leq t\leq 1$. The order interval $[e^{-m}I,e^{m}I]$ is closed under inversion, also $1\leq x^{1/2}\leq x$ for $x\in[1,\infty)$ and $1\geq x^{1/2}\geq x$ for $x\in(0,1)$, so $X_t^{-1/2}AX_t^{-1/2}\in[e^{-m}I,e^{m}I]$ for all $A$ such that $\mu$ is supported in $[0,1]\times\{A\}$. 
By the previous lemmas therefore $X_t^{-1/2}AX_t^{-1/2}\to X_0^{-1/2}AX_0^{-1/2}$ strongly. By the strong continuity of $\l_s$
\begin{equation*}
U(A):=\l_s(X_t^{-1/2}AX_t^{-1/2})\to V(A):=\l_s(X_0^{-1/2}AX_0^{-1/2}).
\end{equation*}
By Lemma~\ref{lemstrongconv} in the strong topology we have
\begin{equation}\label{eq1}
\lim_{t\to 0+}\frac{\l_s^{-1}(tU(A))-I}{t}=V(A)=\l_s(X_0^{-1/2}AX_0^{-1/2}).
\end{equation}

By definition $X_t=\int_{[0,1]\times\mathbb{P}}M_{s,t}(X_t,A)d\mu(s,A)$ which is equivalent to
\begin{eqnarray*}
I&=&\int_{[0,1]\times\mathbb{P}}f_{s,t}(X_t^{-1/2}AX_t^{-1/2})d\mu(s,A)\\
&=&\int_{[0,1]\times\mathbb{P}}\l_s^{-1}\left(t\l_s(X_t^{-1/2}AX_t^{-1/2})\right)d\mu(s,A),
\end{eqnarray*}
that is $0=\int_{[0,1]\times\mathbb{P}}\frac{f_{s,t}(X_t^{-1/2}AX_t^{-1/2})-I}{t}d\mu(s,A)$. By \eqref{eq1} we have
\begin{equation}
\begin{split}
0&=\lim_{t\to 0+}\int_{[0,1]\times\mathbb{P}}\frac{f_{s,t}(X_t^{-1/2}AX_t^{-1/2})-I}{t}d\mu(s,A)\\
&=\int_{[0,1]\times\mathbb{P}}\lim_{t\to 0+}\frac{f_{s,t}(X_t^{-1/2}AX_t^{-1/2})-I}{t}d\mu(s,A)\\
&=\int_{[0,1]\times\mathbb{P}}\l_s(X_0^{-1/2}AX_0^{-1/2})d\mu(s,A),
\end{split}
\end{equation}
where we used Theorem~\ref{T:domconv} and that the function $\frac{f_{s,t}(x)-1}{t}$ is in $\mathfrak{L}$ which follows from basic calculations. By this we also have that $$\int_{[0,1]\times\mathbb{P}}X_0^{1/2}\l_s(X_0^{-1/2}AX_0^{-1/2})X_0^{1/2}d\mu(s,A)=0.$$
\end{proof}

\begin{lemma}\label{karchercongruenceinv}
The set $K(\mu)$ is invariant under congruencies i.e., for any $C\in \mathrm{GL}(E)$
\begin{equation*}
CK(\mu)C^{*}=K(C\mu C^{*}).
\end{equation*}
\end{lemma}
\begin{proof}
For any $X\in K(\sigma)$ we have
\begin{equation}\label{karchereqy}
\begin{split}
0&=\int_{[0,1]\times\mathbb{P}}\l_s(X^{-1/2}AX^{-1/2})d\mu(s,A)\\
&=\int_{[0,1]\times\mathbb{P}}\l_s(X^{-1}A)d\mu(s,A).
\end{split}
\end{equation}
Let $C=UP$ the polar decomposition of $C$ i.e., $U^{-1}=U^{*}$ and $P\in\mathbb{P}$.
Then by \eqref{karchereqy} it follows directly that
\begin{equation*}
UK(\mu)U^{*}=K(U\mu U^{*}).
\end{equation*}
Similarly we have
\begin{equation*}
\begin{split}
0&=P^{-1}\left(\int_{[0,1]\times\mathbb{P}}\l_s(X^{-1/2}AX^{-1/2})d\mu(s,A)\right)P\\
&=\int_{[0,1]\times\mathbb{P}}\l_s(P^{-1}X^{-1}AP)d\mu(s,A)\\
&=\int_{[0,1]\times\mathbb{P}}\l_s(P^{-1}X^{-1}P^{-1}PAP)d\mu(s,A),
\end{split}
\end{equation*}
so $PXP\in K(P\mu P)$ i.e., $PK(\mu)P\subseteq K(P\mu P)$. Also then $K(\mu)\subseteq P^{-1}K(P\mu P)P^{-1}\subseteq K(\mu)$ which means
\begin{equation*}
PK(\mu)P=K(P\mu P).
\end{equation*}
From this and $UK(\mu)U^{*}=K(U\mu U^{*})$ we get that
\begin{equation*}
CK(\mu)C^{*}=K(C\mu C^{*}).
\end{equation*}
\end{proof}

We have already seen that the set $\mathscr{P}([0,1]\times\mathbb{P})$ is a subset of a Banach space equipped with the total variation norm.

\begin{proposition}\label{propunique0}
There exists $\epsilon>0$ such that for $\mu\in\mathscr{P}([0,1]\times\mathbb{P})$ with $\supp\mu\subseteq [0,1]\times\overline{B}_A(\epsilon)$, the equation
$$\int_{[0,1]\times\mathbb{P}}\l_s\left(X^{-1/2}AX^{-1/2}\right)d\mu(s,A)=0$$
has a unique solution in $\overline{B}_A(\epsilon)$ which is $\Lambda(\mu)$.
\end{proposition}
\begin{proof}
We would like to use the Implicit Function Theorem for Banach spaces. First of all the map $$F_{\mu}(X):=F(\mu,X):=\int_{[0,1]\times\mathbb{P}}\l_s(X^{-1/2}AX^{-1/2})d\mu(s,A)$$ maps from the product of two Banach spaces $\mathscr{P}([0,1]\times\mathbb{P})\times S(E)$ to a Banach space $S(E)$, moreover it is $C^{\infty}$. The Fr\'echet derivative of $F_{\mu}$ is a linear map on $S(E)$. Let $\nu\in\mathscr{P}([0,1])$ and $\sigma_I\in\mathscr{P}(\mathbb{P})$ such that $\sigma_I$ is only supported on the singleton $\{I\}$. Then $F_{\nu\times\sigma_I}(X)=\int_{[0,1]}\l_s(X^{-1})d\nu(s)$ so $F_{\nu\times\sigma_I}(I)=0$ and by the property $\l_s'(1)=1$ we have that the Fr\'echet derivative $DF_{\nu\times\sigma_I}[I]=-id_{S(E)}$. Thus by the Implicit Function Theorem (Theorem 5.9 \cite{lang}) there exists an open neighborhood $U$ of $\nu\times\sigma_I$ in $\mathscr{P}([0,1]\times\mathbb{P})$ and a neighborhood $V$ of $I\in\mathbb{P}$, and a $C^{\infty}$ mapping $g:U\mapsto V$ such that $F_{\mu}(X)=0$ if and only if $X=g(\mu)$ for $\mu\in U$, $X\in V$. 
Now if we pick $\epsilon>0$ such that $\overline{B}_I(\epsilon)\subseteq U$ and $\overline{B}_I(\epsilon)\subseteq V$ then the first part of the assertion is proved for $A=I$. The general case for any $A$ follows from Lemma~\ref{karchercongruenceinv} with $C=A^{1/2}$.

The second part of the assertion is a consequence of property (8) in Theorem~\ref{lambdaprop} which yields that $\Lambda(\mu)\in\overline{B}_A(\epsilon)$ which is then identical to $g(\mu)$ by the first part of the assertion.
\end{proof}

\begin{corollary}
Using Lemma~\ref{L:fubini} and Proposition~\ref{propunique0} it follows that the lambda operator mean $\Lambda(\mu)$ is $C^{\infty}$ on small enough neighborhoods of $\supp\mu$ assuming that $\supp\mu\subseteq[0,1]\times\overline{B}_A(\epsilon)$ for small enough $\epsilon>0$ and fixed $A\in\mathbb{P}$.
\end{corollary}

\begin{theorem}\label{uniquegeneralizedkarcher}
$K(\mu)=\{\Lambda(\mu)\}$ for all $\mu\in\mathscr{P}([0,1]\times\mathbb{P})$.
\end{theorem}
\begin{proof}
We start with a
\begin{claim}
For fixed $X\in\mathbb{P}$ and $t\in(0,1]$ define
\begin{equation*}
dM_{s,t}(X,\mu)(s,A):=d\mu(s,X^{1/2}f_{s,t}^{-1}(X^{-1/2}AX^{-1/2})X^{1/2})
\end{equation*}
if there exists $Z\in\mathbb{P}$ such that $f_{s,t}(Z)=X^{-1/2}AX^{-1/2}$, in the other case when there exists no such $Z\in\mathbb{P}$ define $dM_{s,t}(X,\mu)(s,A):=0$. Then the equation
\begin{equation}\label{eq:uniquegeneralizedkarcher}
X=\Lambda(M_{s,t}(X,\mu))
\end{equation}
has a unique solution in $\mathbb{P}$.
\end{claim}
Let us briefly explain the reason behind the introduction of $M_{s,t}(X,\mu)$. The above piecewise definition of $M_{s,t}(X,\mu)$ is needed since the inverse map
\begin{equation}\label{eq:uniquegeneralizedkarcher3}
g^{-1}(A):=X^{1/2}f_{s,t}^{-1}(X^{-1/2}AX^{-1/2})X^{1/2}
\end{equation}
(which is the inverse of of $g(A)=M_{s,t}(X,A)$) may not be well defined on the whole of $\mathbb{P}$. Now the point of our definition is that for fixed $X\in\mathbb{P}$ we would like to consider the one-parameter family of unique solutions $X_u$ of the equations
\begin{equation*}
X_u=\int_{[0,1]\times\mathbb{P}}M_{s,u}\left(X_u,M_{s,t}\left(X,A\right)\right)d\mu(s,A)
\end{equation*}
for $u\in(0,1]$ defining the induced operator means in the sense of \eqref{meanequ31} with respect to the parameter $t\in(0,1]$. Then we take the limit $u\to 0+$ to obtain the corresponding lambda operator means depending on the measure $M_{s,t}(X,\mu)$, which itself depends on $X,t$ and a fixed $\mu\in\mathscr{P}([0,1]\times\mathbb{P})$.

Now let us turn to the proof of the Claim. First of all notice that for $X\in\mathbb{P}$ the function $g^{-1}(A)$ defined by \eqref{eq:uniquegeneralizedkarcher3} is the inverse of $g(A)$, moreover this inverse is well defined and in fact $C^{\infty}$ since $f_{s,t}(x)$ is an operator monotone function hence analytic and strictly monotone increasing. Indeed, by the above $M_{s,t}(X,\mu)$ is well defined. Moreover for fixed $X\in\mathbb{P}$ the set of all $A\in\mathbb{P}$ such that $M_{s,t}(X,\mu)$ is supported in $[0,1]\times\{A\}$ is bounded in $\mathbb{P}$. Now choose a large enough closed set $S\subseteq \mathbb{P}$ such that for all $X,Y\in S$ the support of $M_{s,t}(X,\mu)$ and $M_{s,t}(Y,\mu)$ is included in $[0,1]\times S$. This is clearly possible since we have the bounds according to Lemma~\ref{maxmininp}
\begin{equation*}
\left((1-t)+tx^{-1}\right)^{-1}\leq f_{s,t}(x)\leq (1-t)+tx
\end{equation*}
and can use a similar argument as in the proof of Proposition~\ref{P:uniquesol}.
Now choose a large enough $0<r<\infty$ such that $S\subseteq\overline{B}_I(r)$. Then for all $X,Y\in S$ and for each $A\in\mathbb{P}$ such that $M_{s,t}(X,\mu)$ or $M_{s,t}(Y,\mu)$ is supported in $[0,1]\times\{A\}$, we have $S\subseteq\overline{B}_A(r)$. By Proposition~\ref{P:principalcontr} $g_{s,t}(X)=M_{s,t}(X,A)$ is a strict contraction on $\overline{B}_A(r)$, with contraction coefficient
\begin{equation*}
\rho=\rho_1+\frac{\log\frac{be^{-r}+ae^{2r(1-\rho_1)}}{be^{-r}+a}}{2r},
\end{equation*}
where $\rho_1=\frac{\log\frac{e^{3r}(1-t)+t}{e^{r}(1-t)+t}}{2r}$ and $a=\frac{s(1-t)}{t+s(1-t)}$, $b=\frac{t}{t+s(1-t)}$. Since $t\leq b\leq 1$ and $0\leq a\leq 1-t$, by (*), using a similar argument as we did to obtain \eqref{eq:usingstar}, we get 
\begin{equation*}
\rho\leq\rho_1+\frac{\log\frac{te^{-r}+(1-t)e^{2r(1-\rho_1)}}{te^{-r}+(1-t)}}{2r},
\end{equation*}
which is strictly less then $1$. 
Let $f(X)=\Lambda(M_{s,t}(X,\mu))$. Now by property (6) in Theorem~\ref{lambdaprop} we have
\begin{equation*}
\begin{split}
d_{\infty}&(\Lambda(M_{s,t}(X,\mu)), \Lambda(M_{s,t}(Y,\mu)))\\
&\leq \underset{\mu\text{ is supported on }\{s\}\times\{A\}}{\sup}\{d_{\infty}(M_{s,t}(X,A),M_{s,t}(Y,A))\}\\
&\leq \rho d_\infty(X,Y),
\end{split}
\end{equation*}
i.e., $f(X)$ is a strict contraction on $S$. Moreover by a similar argument to the proof of Proposition~\ref{P:uniquesol} we see that $f(S)\subseteq S$. So by Banach's fixed point theorem \eqref{eq:uniquegeneralizedkarcher} has a unique solution in $S$. Since $S$ is an arbitrary large enough closed subset of $\mathbb{P}$, it follows that \eqref{eq:uniquegeneralizedkarcher} has a unique solution in $\mathbb{P}$. The claim is proved.

Now let $X\in K(\mu)$. We have that $f_{s,t}(x)=\l_s^{-1}(t\l_s(x))\in\mathfrak{m}(t)$ for all $0<t\leq 1$ and $s\in[0,1]$. By Lemma~\ref{maxmininp} we have that
\begin{equation*}
\left((1-t)+tx^{-1}\right)^{-1}\leq f_{s,t}(x)\leq (1-t)+tx
\end{equation*}
which means that there exists a small enough $0<t\leq 1$ such that
\begin{equation*}
f_{s,t}(X^{-1/2}AX^{-1/2})\in \overline{B}_I(\epsilon)
\end{equation*}
for all $A\in\mathbb{P}$ such that $\mu$ is supported in $[0,1]\times\{A\}$. Also it is easy to see due to \eqref{semigrpproperty} that
\begin{equation*}
\begin{split}
0&=t\left[\int_{[0,1]\times\mathbb{P}}\l_s(X^{-1/2}AX^{-1/2})d\mu(s,A)\right]\\
&=\int_{[0,1]\times\mathbb{P}}\l_s(f_{s,t}(X^{-1/2}AX^{-1/2}))d\mu(s,A)\\
&=\int_{[0,1]\times\mathbb{P}}\l_s(X^{-1/2}AX^{-1/2})dM_{s,t}(X,\mu)(s,A)\\
&=\int_{[0,1]\times\mathbb{P}}\l_s(A)d\left(X^{-1/2}M_{s,t}(X,\mu)X^{-1/2}\right)(s,A),
\end{split}
\end{equation*}
in other words
\begin{equation}\label{eq:uniquegeneralizedkarcher2}
0=\int_{[0,1]\times\mathbb{P}}\l_s(Y^{-1/2}AY^{-1/2})d\left(X^{-1/2}M_{s,t}(X,\mu)X^{-1/2}\right)(s,A)
\end{equation}
with $Y=I$. Now by Proposition~\ref{propunique0} we have that $Y=I$ is the unique solution of \eqref{eq:uniquegeneralizedkarcher} on $\overline{B}_I(\epsilon)$ which is $\Lambda(X^{-1/2}M_{s,t}(X,\mu)X^{-1/2})$.
Hence from the definition of the lambda operator mean $\Lambda(\mu)$ we have
\begin{equation*}
I=\Lambda(X^{-1/2}M_{s,t}(X,\mu)X^{-1/2}).
\end{equation*}
By property (3) in Theorem~\ref{lambdaprop} we have that the above is equivalent to
\begin{equation*}
\begin{split}
X&=X^{1/2}\Lambda(X^{-1/2}M_{s,t}(X,\mu)X^{-1/2})X^{1/2}\\
&=\Lambda(M_{s,t}(X,\mu)).
\end{split}
\end{equation*}
Now the Claim implies that the solution of the above equation is unique, hence all solutions $X\in K(\mu)$ must be identical. By Theorem~\ref{karchersatisfied} we have $\Lambda(\mu)\in K(\mu)$, hence $K(\mu)=\{\Lambda(\mu)\}$.
\end{proof}

\begin{corollary}\label{cor:opmean2}
Suppose $\nu\in\mathscr{P}([0,1])$ and $\sigma\in\mathscr{P}(\mathbb{P})$ and assume that $\sigma(X):=(1-w)\delta_{A}(X)+w\delta_{B}(X)$ with $w\in(0,1)$ where $\delta_{A}(X)$ denotes the Dirac delta supported on ${A}$. Then $\Lambda(\nu\times\sigma)$ is an operator mean in the two variables $(A,B)$ i.e., $\Lambda(\nu\times\sigma)\in\mathfrak{M}$.
\end{corollary}
\begin{proof}
By Corollary~\ref{cor:opmean} and Theorem~\ref{inducedconv} the lambda extension is the strong limit of (induced) operator means i.e.,
\begin{equation*}
\Lambda(\nu\times\sigma)=\lim_{t\to 0+}L_{t}(\nu\times\sigma).
\end{equation*}
By Lemma 6.1 in \cite{kubo} the point-wise weak limit of operator means is an operator mean as well, so therefore it follows that the strong limit $\Lambda(\nu\times\sigma)$ of operator means is also an operator mean in the sense of Definition~\ref{symmean}.
\end{proof}

\begin{remark}
Theorem~\ref{uniquegeneralizedkarcher} gives us a tool to solve operator equations that can be written in the form of a generalized Karcher equation
\begin{equation*}
\int_{[0,1]\times\mathbb{P}}\l_s(A)d\mu(s,A)=0.
\end{equation*}
The solution can be calculated by choosing a sequence $t_l\to 0+$ as $l\to\infty$ and then taking the limit
\begin{equation*}
\lim_{l\to\infty}L_{t_l}(\mu)=\Lambda(\mu).
\end{equation*}
\end{remark}

\begin{example}\label{ex:Lambda}
Consider the function $\log\in\mathfrak{L}$. In \cite{bhatia} formula (V.47) says that
$$\log x=\int_{-\infty}^0\frac{1}{\lambda-x}-\frac{\lambda}{\lambda^2+1}d\lambda$$
which is actually
$$\log x=\int_{0}^\infty\frac{\lambda}{\lambda^2+1}-\frac{1}{\lambda+x}d\lambda.$$
By Corollary~\ref{C:intLrepr} and Corollary~\ref{C:lambdaintchar} we get that
$$\log x=\int_{[0,1]}\frac{x-1}{(1-s)x+s}d\nu(s)$$
with $d\nu(s)=ds$.

Let $\sigma\in\mathscr{P}(\mathbb{P})$ and $\nu\in\mathscr{P}([0,1])$. The Dirac delta supported on $\{a\}$ is denoted by $\delta_{\{a\}}(x)$. Consider the generalized Karcher equations 
$$\int_{\mathbb{P}}f(X^{-1/2}AX^{-1/2})d\sigma(A)=0$$
and their solutions:
\begin{eqnarray*}
&d\nu(s)=\delta_{\{0\}}(s),\ \ f(x)=1-x^{-1},\ \ &\Lambda(\nu\times\sigma)=\left(\int_{\mathbb{P}}A^{-1}d\sigma(A)\right)^{-1}\\
&d\nu(s)=\delta_{\{1\}}(s),\ \ f(x)=x-1,\ \ &\Lambda(\nu\times\sigma)=\int_{\mathbb{P}}Ad\sigma(A)\\
&d\nu(s)=ds,\ \ f(x)=\log x,\ \ &\Lambda(\nu\times\sigma)=\tilde{\Lambda}(\sigma),
\end{eqnarray*}
where $\tilde{\Lambda}(\sigma)$ is the original Karcher mean which has been studied first in \cite{lawsonlim1} for positive operators and in the case of positive matrices at many other places, for example see \cite{bhatiaholbrook,bhatiakarandikar,bhatia2,lawsonlim,limpalfia,moakher}. In particular the first mean is the harmonic mean while the second is the arithmetic mean. Notice that all these cases are extensions of these well known means to the case of measures and operators.

There are other recently found means which are all lambda operator means. The log-determinant $\alpha$-divergence means found in \cite{chebbi} for positive matrices are also such with $\nu$ supported over a singleton $\{s\}$ with $s\in[0,1]$.
\end{example}

The arithmetic and harmonic means have special properties.

\begin{proposition}
Let $\mu\in\mathscr{P}([0,1]\times\mathbb{P})$. Then
\begin{equation}
\left(\int_{\mathbb{P}}A^{-1}d\mu(s,A)\right)^{-1}\leq \Lambda(\mu)\leq \int_{\mathbb{P}}Ad\mu(s,A).
\end{equation}
\end{proposition}
\begin{proof}
First of all by Lemma~\ref{maxmininp} we have for all $s,t\in [0,1]$ that
\begin{equation*}
\left((1-t)+tx^{-1}\right)^{-1}\leq f_{s,t}(x)\leq (1-t)+tx,
\end{equation*}
in other words
\begin{equation*}
[(1-t)X^{-1}+tA^{-1}]^{-1}\leq M_{s,t}(X,A)\leq (1-t)X+tA.
\end{equation*}
Now integrate the above inequality with respect to $\mu$ to conclude the assertion using property (4) in Theorem~\ref{lambdaprop} with Theorem~\ref{karchersatisfied} and the previous Example~\ref{ex:Lambda}.
\end{proof}

One might wonder whether the induced means $L_t(\mu)$ are different from lambda operator means. It turns out that the induced means fit into the overall picture.

\begin{theorem}\label{T:inducedlambda}
Let $t\in(0,1]$ and $\mu\in\mathscr{P}([0,1]\times\mathbb{P})$. Then $L_t(\mu)$ is a lambda operator mean, it is the unique solution of the generalized Karcher equation
$$\int_{[0,1]\times\mathbb{P}}\log_{t+s(1-t)}(X^{-1/2}AX^{-1/2})d\mu(s,A)=0.$$
\end{theorem}
\begin{proof}
By definition \eqref{meanequ3}, $L_t(\mu)$ is the unique solution of
\begin{equation*}
X=\int_{[0,1]\times\mathbb{P}}M_{s,t}(X,A)d\mu(s,A).
\end{equation*}
So we have
\begin{eqnarray*}
X&=&\int_{[0,1]\times\mathbb{P}}M_{s,t}(X,A)d\mu(s,A)\\
0&=&\int_{[0,1]\times\mathbb{P}}M_{s,t}(X,A)-Xd\mu(s,A)\\
0&=&\int_{[0,1]\times\mathbb{P}}\frac{M_{s,t}(X,A)-X}{t}d\mu(s,A)\\
0&=&\int_{[0,1]\times\mathbb{P}}\frac{f_{s,t}(X^{-1/2}AX^{-1/2})-I}{t}d\mu(s,A).
\end{eqnarray*}
Now it is a routine calculation to see that $\frac{f_{s,t}(x)-1}{t}=\log_{t+s(1-t)}(x)$ and $\log_{t+s(1-t)}\in\mathfrak{L}$, since $t+s(1-t)\in[0,1].$
\end{proof}

The above result is a general phenomenon:

\begin{proposition}
Let $f\in\mathfrak{m}(t)$ for $t\in(0,1)$. Then $\frac{f(x)-1}{t}$ is in $\mathfrak{L}$.
\end{proposition}
\begin{proof}
Simple computation.
\end{proof}

\begin{corollary}\label{C:powermeans}
Let $\sigma\in\mathscr{P}(\mathbb{P})$ and let $f_t\in\mathfrak{m}(t)$ be the representing function of the operator mean $M_t(X,A)$ with $t\in(0,1)$. Then
\begin{equation}\label{eq:geninduced}
X=\int_{\mathbb{P}}M_t(X,A)d\sigma(A)
\end{equation}
has a unique solution in $\mathbb{P}$, moreover it is the unique solution of the generalized Karcher equation of the form
\begin{equation*}
\int_{\mathbb{P}}X^{1/2}g\left(X^{-1/2}AX^{-1/2}\right)X^{1/2}d\sigma(A)=0
\end{equation*}
where $g(x)=\frac{f_t(x)-1}{t}$ is in $\mathfrak{L}$.
\end{corollary}
\begin{proof}
Simple computation as in the proof of Theorem~\ref{T:inducedlambda}.
\end{proof}

Corollary~\ref{C:powermeans} tells us that the matrix power means $P_t(\omega;{\Bbb A})$ defined as the unique positive definite solution of \eqref{powermeanequ2} are actually lambda operator means for $t\in[-1,1]$ i.e., they are unique solutions of generalized Karcher equations. Let us provide the details in the following example.

\begin{example}\label{ex:Lambda2}
For $t\in(0,1]$, the matrix power means $P_t(\omega;{\Bbb A})$ are the unique positive definite solutions of
\begin{equation}\label{eq:expowermean}
X=\sum_{i=1}^{k}w_{i}X\#_tA_i
\end{equation}
where
\begin{equation*}
A\#_tB=A^{1/2}\left(A^{-1/2}BA^{-1/2}\right)^tA^{1/2}
\end{equation*}
is the weighted geometric mean of $A,B\in\mathbb{P}$, its representing function is $f_t(x)=x^t$ and $f_t\in\mathfrak{m}(t)$. For $\sigma\in\mathscr{P}(\mathbb{P})$ let us consider a generalized form of \eqref{eq:expowermean} in the form
\begin{equation}\label{eq:expowermean2}
X=\int_{\mathbb{P}}X\#_tAd\sigma(A).
\end{equation}
By Corollary~\ref{C:powermeans} this is equivalent to
\begin{equation}\label{eq:expowermeankarcher}
0=\int_{\mathbb{P}}X^{1/2}f\left(X^{-1/2}AX^{-1/2}\right)X^{1/2}d\sigma(A)
\end{equation}
where $f(x):=\frac{x^t-1}{t}$ and $f\in\mathfrak{L}$. Now using the representation V.48 in \cite{bhatia} for the power function $x^t$, we have that
\begin{eqnarray*}
f(x)=\frac{x^t-1}{t}&=&\frac{\cos\frac{t\pi}{2}-1}{t}+\int_{-\infty}^{0}\left(\frac{1}{\lambda-x}-\frac{\lambda}{\lambda^2+1}\right)|\lambda|^t\frac{\sin(t\pi)}{t\pi}d\lambda\\
&=&\frac{\cos\frac{t\pi}{2}-1}{t}+\int_{0}^{\infty}\left(\frac{\lambda}{\lambda^2+1}-\frac{1}{\lambda+x}\right)\lambda^t\frac{\sin(t\pi)}{t\pi}d\lambda.
\end{eqnarray*}
After applying change of variables and using Corollary~\ref{C:intLrepr} and Corollary~\ref{C:lambdaintchar} we get that
$$f(x)=\frac{x^t-1}{t}=\int_{0}^{1}\frac{x-1}{(1-s)x+s}d\nu(s)$$
where
$$d\nu(s)=\frac{s^t}{(1-s)^t}\frac{\sin(t\pi)}{t\pi}ds.$$
Hence by \eqref{eq:expowermeankarcher} the matrix power means $P_t$ are unique solutions of the generalized Karcher equations
\begin{equation}\label{eq:expowermeankarcher2}
0=\int_{\mathbb{P}}\int_{0}^{1}X^{1/2}\l_s\left(X^{-1/2}AX^{-1/2}\right)X^{1/2}d\nu(s)d\sigma(A).
\end{equation}
In the finite dimensional case one can prove more by applying Theorem~\ref{finalconclusion} in the next section to conclude that the matrix power means are global minimizers of geodesically convex functions. Similar calculations can be carried out for negative values of $t\in[-1,0)$. Also by taking the limit $t\to 0$ we obtain the case of the original Karcher mean considered in Example~\ref{ex:Lambda}.
\end{example}

\section{Conclusion}
The previous sections provide the completion of Theorem~\ref{T:optimize} in the case when $\mathbb{P}$ is the finite dimensional cone of positive matrices. Theorem~\ref{uniquegeneralizedkarcher} tells us that the gradient equations
\begin{equation}
\int_{\mathfrak{L}\times\mathbb{P}}X^{1/2}f(X^{-1/2}AX^{-1/2})X^{1/2}d\sigma(f,A)=0
\end{equation}
given in \eqref{eq:optimize2} for $\sigma\in\mathscr{P}(\mathfrak{L}\times\mathbb{P})$ admit unique solutions in $\mathbb{P}$. Theorem~\ref{T:optimize} shows us that
\begin{equation}
\int_{[0,1]\times\mathbb{P}}X^{1/2}\l_s(X^{-1/2}AX^{-1/2})X^{1/2}d\mu(s,A)=0
\end{equation}
has a unique solution and now it is not difficult to see that it is the gradient equation of the critical points of the strictly geodesically convex functional
\begin{equation*}
\int_{[0,1]\times\mathbb{P}}LD^s(X,A)d\mu(s,A).
\end{equation*}
This argument essentially can be carried out using Theorem~\ref{T:RiemHess} and Theorem~\ref{T:RiemaGrad}. Hence we have the result:
\begin{theorem}\label{finalconclusion}
Let $\mathbb{P}$ be finite dimensional and $\mu\in\mathscr{P}([0,1]\times\mathbb{P})$. Then
\begin{equation}
\Lambda(\mu)=\argmin_{X\in\mathbb{P}}\int_{[0,1]\times\mathbb{P}}LD^s(X,A)d\mu(s,A).
\end{equation}
\end{theorem}

\begin{remark}
More questions can be asked about the size of the set generated by $\Lambda(\mu)$ for various different $\mu\in\mathscr{P}([0,1]\times\mathbb{P})$. Perhaps Corollary~\ref{cor:opmean2} provides the correct setting. How large is the set of two-variable lambda operator means in $\mathfrak{M}$? This question is nontrivial and it is very likely that there are operator means in $\mathfrak{M}$ which are not lambda operator means. In particular if $w=1/2$ in Corollary~\ref{cor:opmean2} we have some examples when a two variable operator mean is the unique solution of two different generalized Karcher equations. The geometric mean $A\#_{1/2}B=A^{1/2}(A^{-1/2}BA^{-1/2})^{1/2}A^{1/2}$ is such, see \cite{chebbi}. However this redundancy only occurs if $w=1/2$. This result, though might be important, is out of scope of this paper and will be presented elsewhere. So all in all the set of lambda operator means is of the "size" of $\mathfrak{L}$ if $w\neq 1/2$. In this case we even assumed that the same $f\in\mathfrak{L}$ is considered in the generalized Karcher equation i.e., $\mu=\nu\times\sigma$ is a product. In the general case we expect the set of lambda operator means to be even larger.
\end{remark}

\section*{Acknowledgment}
The author would like to thank the anonymous referee for valuable comments and suggestions. This work was partly supported by the Research Fellowship of the Canon Foundation, SGU project of Kyoto University, the JSPS international research fellowship grant No. 14F04320, the "Lend\"ulet" Program (LP2012-46/2012) of the Hungarian Academy of Sciences and the National Research Foundation of Korea (NRF) grant funded by the Korea government (MEST) No. 2015R1A3A2031159.




\end{document}